\documentclass[12pt,leqno]{article} 
\usepackage{amsmath,amstext,amsthm,amsfonts,amssymb,amsthm}  
  \numberwithin{equation}{section} 
\def\R{\mathbb{ R}}  
  
\usepackage{epsfig}
\newtheorem{thm}{Theorem}[section]
\newtheorem{lem}[thm]{Lemma}
\renewcommand{\O}{\mathcal{O}}

\newcommand{\be}{\begin{equation}}
\newcommand{\ee}{\end{equation}}

\newcommand{\ba}{\begin{array}}
\newcommand{\ea}{\end{array}}

\newcommand{\bg}{\begin{gathered}}
\newcommand{\eg}{\end{gathered}}

\renewcommand{\a}{\alpha}

\renewcommand{\t}{\theta}
\renewcommand{\O}{\mathcal{O}}
\renewcommand{\l}{\lambda}
\newcommand{\f}{\phi}
\renewcommand{\t}{\theta}

\newcommand{\A}{\mathcal{A}_q}
\newcommand{\aw}{Askey-Wilson  }

\newcommand{\bea}{\begin{eqnarray}}
\newcommand{\eea}{\end{eqnarray}}
\newcommand{\ac}{Al-Salam--Chihara }
\newcommand{\D}{\mathcal{D}_q}
\newcommand{\Sum}{\sum_{n=0}^\infty}

\begin{document}

\title{   Orthogonal Polynomials of \aw Type}
\author{Mourad E.H. Ismail \and Ruiming Zhang \and Keru Zhou}

\maketitle

\begin{abstract}
We study two families of orthogonal polynomials. The first is a finite 
family related to the Askey--Wilson polynomials but the orthogonality is on $\R$. A limiting case of this family is an infinite system of orthogonal polynomials whose moment problem is indeterminate. We provide several orthogonality measures for the infinite family and derive their Plancherel-Rotach asymptotics. 
\end{abstract}

Filename: IsmZhangZhouNEWOPV5.tex




\section{Introduction}

This paper outgrew from the first named author's earlier paper 
\cite{Ism2020} where some solutions to the 
\ac moment problem was found.  We started with the weight 
function whose total mass was evaluated by Askey in \cite{Ask2} and 
we were led to \aw polynomials with purely imaginary parameters, 
which are not necessarily pairs of complex conjugates. 
 After a change of variable our  polynomials  form finite families of 
 polynomials orthogonal on $\R$. It turned out that Askey has already 
partially  discovered this fact in \cite{Ask2}. 
The details are in Section 2. 
Section 2 also  contains raising and 
 lowering   operators for the finite family of  polynomials as well as the 
 second order operator equation satisfied by them.  The second order 
 operator equation  is of Sturm--Liouville type and are selfadjoint (symmetrical).  Although the 
 orthogonality holds for finitely many polynomials the polynomials 
 are defined for all degrees. We determined the large degree 
 asymptotics which shows that  the zeros of the polynomials form a dense set in the segment connecting $\pm i$.  
 
When we further let one of the four 
parameters tend  to zero,  we have an infinite family of polynomials 
orthogonal on the imaginary axis with respect to infinitely many 
probability measures.  We identify one absolutely  continuous  
measure and an infinite family of discrete measures of orthogonality.  
This is done in Section 3. 
In Section 4 we   derive Plancherel--Rotach type asymptotics   around the largest zero (soft edge) and 
beyond the largest zero (tail).  We also develop  the large degree  
asymptotics of the polynomials in the oscillatory range (bulk 
scaling).  In addition, we develop a new type of asymptotics, where 
we let the parameters also tend to $\infty$ with 
$x$ around the largest zero. In this limit the leading terms of 
the asymptotics of the zeros, arranged from large to small, contains the 
zeros of the Ramanujan function.  The Plancherel-Rotach asymptotics of the $q^{-1}$-Hermite polynomials, the Stieltjes--Wigert polynomials and the $q$-Laguerre polynomials are in 
\cite{Ism2005}, and \cite{Ism:Zha}. The weight function given in Section 2 is not positive when the parameters are real. In Section 5 we treat the case of the finite family when the parameters are not real but are complex conjugates. This leads to positive weight functions.  

This work extends the results of Ismail \cite{Ism2020}, where he studied the moment problem of the Al-Salam--Chihara polynomials for $q > 1$.  The \ac polynomials first appeared in \cite{Als:Chi}. 
The $q > 1$ cases were first studied in \cite{Ask:Ism}.
We follow the treatments of the  moment problem in \cite{Akh}, \cite{Sho:Tam}, and the spectral theory as in \cite{Tes}. This work is a contribution to the study of specific  moment problems. Many 
other moment problems have been studied over the years. Some  references are \cite{Ism:Rah}, \cite{Ber:Chr}, \cite{Chr:Koe}, \cite{Chr2}, \cite{Chi2}, 
\cite{Chr:Ism}. The most complete study is the $q^{-1}$-Hermite polynomials where theta functions made it possible to explicitely find, among other things,  the $N$-extremal measures. References for orthogonal polynomials are 
\cite{Chi}, \cite{Koe:Swa}.  The operator equations derived in \S 2 
extend the work of Ismail \cite{Ism93} on the $q^{-1}$-Hermite polynomials. 

The Ramanujan, aka $q$-Airy function 
\bea
\label{eqAq}
A_q(z) = \Sum \frac{q^{n^2}}{(q;q)_n} \, (-z)^n. 
\eea
was introduced in \cite{Ism2005}, see also \cite{Ismbook}. In many of our computations we shall use \cite{Gas:Rah},
\bea
(aq^{-n};q)_n = (q/a;q)_n(-a)^n q^{-\binom{n+1}{2}}.
\eea 

\section{A Finite Family of Orthogonal Polynomials} 
We shall use the notation  
\begin{equation}
\label{eqxofz}
x=\frac{z-1/z}{2}.
\end{equation}
When we set $2x = z-1/z$, then $z = x\pm \sqrt{x^2+1}$. We shall use the notation 
\bea
\label{eqdfz}
z, 1/z  = x\pm \sqrt{x^2+1}, \quad \textup{with}\quad |z|\le |1/z|. 
\eea
Set  
\bea
u_n(x;a) = (-aq^{-n}z, aq^{-n}/z;q)_n.
\eea
Recall the Askey $q$-beta integral \cite{Ask2}, \cite[Ex 6.10]{Gas:Rah}
\begin{align}
\label{eqAsketqbeta}
&
I(t_1, t_2, t_3, t_4):= 
\int_\R \frac{2z\prod_{j=1}^4 (-t_jz,  t_j/z;q)_\infty}
{(-z^2, -q/z^2;q)_\infty}dx  
\\ 
&= -\log q \; (q;q)_\infty\frac{\prod_{1 \le j < k \le 4}(-t_jt_k/q;q)_\infty}
{(t_1t_2t_3t_4/q^3;q)_\infty}, \nonumber
\end{align}
which holds for $|t_1t_2t_3t_4| < q^3$. 
Let 
\bea
\label{eqdefW}
W(x, {\bf t}) = \frac{2z\prod_{j=1}^4 (-t_jz,  t_j/z;q)_\infty}
{(-z^2, -q/z^2;q)_\infty},
\eea
where ${\bf t} = (t_1, t_2, t_3, t_4)$ and $x\in \mathbb{R}$ is defined by \eqref{eqxofz}.  The polynomials defined below in 
\eqref{eqDefASP} were introduced by Askey in his seminal work 
\cite{Ask2} who proved the orthogonality relation in Theorem 
\ref{thm:finite}. We will include Askey's proof because it will used 
to prove other orthogonality relations for the same polynomials. 
The proof is analogous to the attachment technique used 
by Askey and Wilson in \cite{Ask:Wil} and others.  The \aw proof is also explained in \cite{Ismbook}.

\begin{thm}
\label{thm:finite}Given any $N\in\mathbb{N}$, let
\begin{equation}
t_{1},\,t_{2},\,t_{3},\,t_{4}\in\mathbb{R},\quad\left|t_{1}t_{2}t_{3}t_{4}q^{3}\right|<q^{2N}.\label{eq:finite-condition}
\end{equation}
Then the polynomials, 
\begin{align}
&
p_{n}(x,\,\mathbf{t})=\left(t_{1}/q\right)^{n}\left(-q^{2}/t_{1}t_{2},
-q^{2}/t_{1}t_{3},-q^{2}/t_{1}t_{4};q\right)_{n}\label{eqDefASP}\\
&
 \times{}_{4}\phi_{3}\begin{pmatrix}\begin{array}{c}
q^{-n},\,q^{n+3}/t_{1}t_{2}t_{3}t_{4},-q/t_{1}z,\,qz/t_{1}\\
-q^{2}/t_{1}t_{2},-q^{2}/t_{1}t_{3},-q^{2}/t_{1}t_{4}
\end{array} & \bigg|q,\,q\end{pmatrix} 
\nonumber
\end{align}
are orthogonal with respect to the normalized weight function, 
\begin{align}
& w(x,\,\mathbf{t})=\frac{W(x,\,\mathbf{t})(t_{1}t_{2}t_{3}t_{4}/q^{3};q)_{\infty}}{(q;q)_{\infty}\log q^{-1}\prod_{1\le j<k\le4}\left(-t_{j}t_{k}/q;q\right)_{\infty}}\label{eq:finite-weight}\\
	& =\frac{2z(t_{1}t_{2}t_{3}t_{4}/q^{3};q)_{\infty}\prod_{j=1}^{4}(-t_{j}z,t_{j}/z;q)_{\infty}}{(q,-z^{2},-q/z^{2};q)_{\infty}\log q^{-1}\prod_{1\le j<k\le4}\left(-t_{j}t_{k}/q;q\right)_{\infty}}\nonumber 
	\end{align}
for $0\le n\le N$ where ${\bf t}=(t_{1},t_{2},t_{3},t_{4})$ and $x\in \mathbb{R}$ is defined by \eqref{eqxofz}. Furthermore,
the orthogonality relation is 
\begin{align}
&
 \int_{\mathbb{R}}w(x,\mathbf{t})p_{n}(x,\,\mathbf{t})\overline{p_{m}(x,\,\mathbf{t})}dx \label{eq:finite-orthogonality} \\
 &
 =\frac{(-1)^n(1-q^{n+3}/t_{1}t_{2}t_{3}t_{4})    \prod_{1 \le j < k \le 4}\left(-q^{2}/t_{j}t_{k};q\right)_{n}  (q;q)_n}{(1-q^{2n+3}/t_{1}t_{2}t_{3}t_{4})(q^{4}/t_{1}t_{2}t_{3}t_{4};q)_{n}}
\delta_{m,n}.
\nonumber
\end{align}
\end{thm}

Before proving Theorem \ref{thm:finite} we next indicate the range of the parameters to ensure orthogonality. 
To determine the large $z$ behavior of $W$ we set $z = q^{-m}\l$, where $1<  | \l |  \le 1/q$. For this $x$ we have 
\begin{align*}
&
W(x; {\bf t}) = \O\left( q^{-m} \frac{\prod_{j=1}^4 (-t_j \l q^{-m};q)_m}
{(-\l^2q^{-2m};q)_{2m}} \right) = \O \left((t_1t_2t_3t_4q^{-2})^m\right). 
\end{align*}
For integrability we  need $W(x; {\bf t})$ to be $\O(x^{-1-\epsilon})$ for some positive 
$\epsilon$. This happens  if and  only if $|t_1t_2t_3t_4q^{-3}| < 1$. The moments $\int_\R x^n W(x, {\bf t}) dx$ exist for $0\le n \le 2N$, if 
$|t_1t_2t_3t_4q^{-3}| < q^{2N}$.

\begin{proof}[Proof of Theorem \textup{\ref{thm:finite}}] 
 Let 
\bea
p_n(x, {\bf t}) = \sum_{k=0}^n \frac{(q^{-n};q)_k}{(q;q)_k} a_{n,k}
\; u_k(x;t_1). 
\notag
\eea
where $a_{n.k}$ are to be determined. 
We now consider the integral
\begin{align*}
	&
I_{m,k} := \int_\R W(x, {\bf t})  u_k(x;t_1)u_m(x; t_2) \; dx,\quad 0\le m,k\le N.
\end{align*}
It is clear that $I_{m,k} = I(t_1q^{-k}, t_2q^{-m}, t_3, t_4)$, hence for $0\le n,m\le N$,
\begin{align*}
&
\int_\R W(x, {\bf t})  p_n(x, {\bf t}) u_m(x; t_2) dx \\
&= I(t_1, t_2q^{-m}, t_3, t_4)  \sum_{k=0}^n \frac{(q^{-n};q)_k}{(q;q)_k} 
 a_{n,k} \frac{(-t_1t_2q^{-m-k-1}, -t_1t_3q^{-k-1}, -t_1t_4q^{-k-1};q)_k}
{(t_1t_2t_3t_4q^{-m-k-3};q)_k}\\
&= I(t_1, t_2q^{-m}, t_3, t_4)  \sum_{k=0}^n \frac{(q^{-n};q)_k}{(q;q)_k} 
\frac{(-q^{m+2}/t_1t_2, -  q^2/t_1t_3, -q^2/t_1t_4;q)_k}
{(q^{m+4}/t_1t_2t_3t_4;q)_k} t_1^{2k}q^{-k(k+1)}  (-1)^k a_{n,k}. 
\end{align*}
We now choose 
\bea
\notag
a_{n,k} =C_n \frac{(q^{n+3}/t_1t_2t_3t_4;q)_k}
{(-q^2/t_1t_2, -q^2/t_1t_3, -q^2/t_1t_4;q)_k} q^{k(k+2)}(-1)^k t_1^{-2k}, 
\eea
and conclude that 
\begin{align*}
&\int_\R W(x, {\bf t})  p_n(x, {\bf t}) u_m(x; t_2) dx \\ 
&= C_nI(t_1, t_2q^{-m}, t_3, t_4) \; 
 {}_3\phi_2\left(\left. \ba{c} q^{-n}, q^{n+3}/t_1t_2t_3t_4, -q^{m+2}/t_1t_2 \\
q^{m+4}/t_1t_2t_3t_4, -q^2/t_1t_2  \ea \right|q , q \right)\\
&=  C_n I(t_1, t_2q^{-m}, t_3, t_4)\;  \frac{(q^{m+1-n}, q^{2}/t_3t_4;q)_n}
{(q^{m+4}/t_1t_2t_3t_4, q^{-n-1}t_1t_2;q)_n}, 
\end{align*}
which clearly vanishes for $m < n$. We choose $C_n$ to be the factor 
in front of the ${}_4\f_3$ in \eqref{eqDefASP}. It is clear that the integral 
in \eqref{eq:finite-orthogonality} equals 
\begin{align*}
&t_1^n\frac{(q^{-n}, q^{n+3}/t_1t_2t_3t_4;q)_n}{(q;q)_n}  \int_\R 
p_n(x; {\bf t}) W(x; {\bf t}) (-q/t_1z, qz/t_1;q)_n dx \\
&= t_2^n\frac{(q^{-n}, q^{n+3}/t_1t_2t_3t_4;q)_n}{(q;q)_n}  \int_\R 
p_n(x; {\bf t}) W(x; {\bf t}) (-q/t_2z, qz/t_2;q)_n dx\\
&= \frac{(-1)^n(1-q^{n+3}/t_{1}t_{2}t_{3}t_{4})    \prod_{1 \le j < k \le 4}\left(-q^{2}/t_{j}t_{k};q\right)_{n}  (q;q)_n}{(1-q^{2n+3}/t_{1}t_{2}t_{3}t_{4})(q^{4}/t_{1}t_{2}t_{3}t_{4};q)_{n}}I(t_1,t_2,t_3,t_4).
\end{align*}
This completes the proof. 
\end{proof}

It must be noted that the weight function $W$ is not positive on $\R$ if the parameters $t_1, t_2, t_3, t_4$ are real and distinct but $W$ is positive when the parameters form a pair of complex conjugates.  

The orthogonality measure of the $q^{-1}$-Hermite polynomials is not unique.   Askey \cite{Ask1} identified the weight function, \cite{Ismbook}
\begin{equation}
	\label{akf}
	w_A(x):=  \frac{-2z/\log q}{(q, -z^2,-q/z^2;q)_\infty}, \quad   
	x=(z-1/z)/2,
\end{equation}
for the $q^{-1}$-Hermite polynomials,  here $\int_\R w_A(x) dx =1$. 
Ismail and Masson \cite{Ism:Mas} 
proved that the $q^{-1}$-Hermite polynomials are orthogonal with respect to  a family of discrete measures supported at the sequences of points $\{x_n(\alpha):-\infty<n<+\infty\}$ with the masses $m_n(\alpha)$ at $x_n(\alpha)$, where 
\begin{equation}\label{mass}
x_n(\alpha)=(q^{-n}/\alpha-\alpha q^n)/2, \quad  m_n(\alpha)=\frac{\alpha^{4n}q^{n(2n-1)}(1+\alpha^2q^{2n})}{(-\alpha^2,-q/\alpha^2,q;q)_\infty},
\end{equation}
with $\alpha\in(q,1)$, see more details in \cite{Ismbook} and \cite{Ism:Mas}.  These measure are the only measures which make the polynomials dense in their weighted $L_2$ spaces. They are normalized to have total mass $1$.

We denote the measure in \eqref{mass}  by 
$\mu_\a$ and  define a measure $\mu$ by 
\bea
\label{eqortmua}
\mu(x) = \frac{(t_{1}t_{2}t_{3}t_{4}/q^{3};q)_{\infty} \prod_{j=1}^4 (-t_jz,  t_j/z;q)_\infty}
{\prod_{1\le j<k\le4}\left(-t_{j}t_{k}/q;q\right)_{\infty}}\; \mu_\a(x),
\eea
with $x = (z-1/z)/2$.  Then under the conditions in Theorem \ref{thm:finite} we have the orthogonality relation
\begin{align}
&
\int_{\mathbb{R}} p_{n}(x,\,\mathbf{t})\overline{p_{m}(x,\,\mathbf{t})}
d\mu(x) \label{eqorR}  \\
& =\frac{(-1)^n(1-q^{n+3}/t_{1}t_{2}t_{3}t_{4})   
  \prod_{1 \le j < k \le 4}\left(-q^{2}/t_{j}t_{k};q\right)_{n} 
 (q;q)_n}{(1-q^{2n+3}/t_{1}t_{2}t_{3}t_{4})
  (q^{4}/t_{1}t_{2}t_{3}t_{4};q)_{n}}
\delta_{m,n}.
\nonumber
\end{align}
The proof is the same as the proof of Theorem \ref{thm:finite} once 
we evaluate the total mass of $\mu$. The total mass of $\mu$ is 
evaluated using the ${}_6\psi_6$ summation theorem as 
indicated in \cite{Ism:Mas}.  We note that the only difference 
between  \eqref{eqorR} and \eqref{eq:finite-orthogonality} is that 
the normalized Askey measure $w_A$ is replaced  by $\mu_\a$.

 It may of interest to explain why $\mu_\a$ is a discrete version of 
 the Askey weight function $w_A$. The parameterization used is 
 $x_n(\a) = (q^{-n}/\a- \a q^n)/2$, so that 
 \begin{equation*}
 dx = \frac{q^{-n}/\a+  \a q^n}{2} (-\log q) dn,
 \end{equation*}
  and we interpret $dn$ as the mesh used which is 1 in tis case. 
At the same time, we find the Askey weight function in \eqref{akf} has the property as follow:		
\begin{align*}
&
-\log q\;  w_{A}\left(x_{n}(\alpha)\right) =\frac{ 2q^{-n} / \alpha}{\left(q, -q^{-2 n} / \alpha^{2},-q^{2 n+1} \alpha^{2} ; q\right)_{\infty}} \\
&=  \frac{q^{-n}\left(-q \alpha^{2} ; q\right)_{2 n}}
{\alpha\left(-q^{-2 n} / \alpha^{2} ; q\right)_{2 n}} \frac{2} 
{\left( q, -1 / \alpha^{2},-q \alpha^{2} ; q\right)_{\infty}}
 =\frac{2 \alpha^{4 n+1} q^{2 n^{2}}}{\left(-q / \alpha^{2},
 -\alpha^{2} ; q\right)_{\infty}}.
\end{align*}
Hence 
\begin{eqnarray}\label{askfp2}
m_n(\alpha)=w_A(x_n(\alpha))dx_n(\alpha). 
\end{eqnarray}	
It implies that any $N$-extremal measure is essentially equal to the value of the Askey weight functions multiplied by $dx$ calculated at the mass points.

Recall that the \aw polynomials \cite{Ask:Wil}, \cite{Gas:Rah}, 
\cite{Ismbook} are defined by 
\begin{align}
&
 AW_n(\cos \t,\,{\bf t}) = 
  \frac{(t_{1}t_{2}, t_{1}t_{3},t_{1}t_{4};q)_n}{t_1^n} \label{eqAWn} \\  &\times
{}_4\phi_{3}\left(\left. \begin{array}{c}
q^{-n},\,q^{n-1}t_{1}t_{2}t_{3}t_{4}, t_{1}e^{i\t}, t_{1}e^{-i\t}\\
t_{1}t_{2}, t_{1}t_{3}, t_{1}t_{4} \end{array}
  \right| q,\,q\right).  
\nonumber
\end{align}
We note that the algebraic properties of the polynomials 
$\{p_n(x; {\bf t})$ follow from the corresponding properties of 
$\{AW_n(x; {\bf t})$ by using the simple rules 
\begin{equation}
\label{eq2.10}
t_j \to iq/t_j, 1 \le j \le 4, \quad e^{i\t} \to -iz,  \quad  AW_n(\cos \t,\,{\bf t}) 
\to i^n   p_n(x; {\bf t}).
\end{equation}
For example the Ismail--Wilson generating function for the Askey-
Wilson polynomials in \cite{Ism:Wil}, which is reproduced in \cite{Gas:Rah}, \cite{Ismbook}, 
implies the generating function 
\begin{align}
&
\sum_{n=0}^{\infty}\frac{p_{n}(x,\,\mathbf{t})t^{n}}{\left(q,-q^{2}/t_{1}t_{2},-q^{2}/t_{3}t_{4};q\right)_{n}}\label{eqgfnewpol} \\
&= {} _{2}\phi_{1}\begin{pmatrix}\begin{array}{c}
qz/t_{1},\:qz/t_{2}\\
-q^{2}/t_{1}t_{2}
\end{array} & \bigg|q,\,-t/z\end{pmatrix}{}_{2}\phi_{1}\begin{pmatrix}\begin{array}{c}
-q/zt_{3},\:-q/zt_{4}\\
-q^{2}/t_{3}t_{4}
\end{array} & \bigg|q,\,tz\end{pmatrix}.
\nonumber
\end{align}
The symmetry of $p_{n}(x,\,\mathbf{t})$ under 
$t_j \leftrightarrow t_k, 1\le j, k \le 4$ gives additional generating functions. .

The connection relation for our polynomials follow from the corresponding results for the Askey--Wilson polynomials  in \cite{Ask:Wil}, \cite{Gas:Rah}. In particular
\begin{equation}
p_{n}(x,\,\mathbf{s})=\sum_{k=0}^{n}c_{k,n}\left(\mathbf{s},\,\mathbf{t}\right)p_{k}(x,\,\mathbf{t}),\label{eq:finite-connection-1}
\end{equation}
where 
\begin{equation}
\mathbf{t}=t_{1},\ t_{2},\ t_{3},\ t_{4}\label{eq:finite-connection-2}
\end{equation}
and
\begin{equation}
\mathbf{s}=s_{1},\ s_{2},\ s_{3},\ s_{4},\quad s_{4}=t_{4}.\label{eq:finite-connection-3}
\end{equation}
then by the connection coefficient problem for the Askey-Wilson problem
in \cite{Gas:Rah} we get
\begin{align}
& c_{k,n}\left(\mathbf{s},\,\mathbf{t}\right)=\frac{\left(-q^{2}/s_{1}t_{4},-q^{2}/s_{2}t_{4},-q^{2}/s_{3}t_{4},\,q;q\right)_{n}\left(q^{n+3}/s_{1}s_{2}s_{3}t_{4};q\right)_{k}q^{k^{2}-nk}}{\left(-q^{2}/s_{1}t_{4},-q^{2}/s_{2}t_{4},-q^{2}/s_{3}t_{4},\,q,\,q^{k+3}/t_{1}t_{2}t_{3}t_{4};q\right)_{k}\left(q;q\right)_{n-k}}\left(\frac{q}{t_{4}}\right)^{k-n}\label{eq:finite-connection-4}\\
& \times{}_{5}\phi_{4}\begin{pmatrix}\begin{array}{c}
q^{k-n},\,q^{n+k+3}/s_{1}s_{2}s_{3}t_{4},-q^{k+2}/t_{1}t_{4},-q^{k+2}/t_{2}t_{4},-q^{k+2}/t_{3}t_{4}\\
q^{2k+4}/t_{1}t_{2}t_{3}t_{4},-q^{k+2}/s_{1}t_{4},-q^{k+2}/s_{2}t_{4},-q^{k+2}/s_{3}t_{4}
\end{array} & \bigg|q,\ q\end{pmatrix}.\nonumber 
\end{align}

Our next task is to identify  raising and lowering operators for our polynomials. 
The Askey--Wilson 
operator $\D$ and the averaging operator $\A$ are  defined by 
\begin{eqnarray}\label{asko}
\bg
  (\D f)(x) = \frac{\breve{f}(q^{1/2}z)- \breve{f}(q^{-1/2}z)}
{(q^{1/2}- q^{-1/2})(z+1/z)/2}, \\
  (\A f)(x) =\frac{1}{2}
\left[\breve{f}(q^{1/2}z)+ \breve{f}(q^{-1/2}\right].
\eg
\end{eqnarray}
A calculation gives 
\bea
\D (-a/z, az;q)_k =- 2a \frac{1-q^k}{1-q} (-aq^{1/2}/z, aq^{1/2}z;q)_{k-1}
\eea
Therefore 
\bea
\label{eqlowering}
\D p_n(x, {\bf t}) = \frac{2(1-q^{n})(1-q^{n+3}/t_1t_2t_3t_4)}
{(1-q)q^{(n-1)/2}} p_{n-1}(x, q^{-1/2}{\bf t}) .
\eea

 We can also  establish the  raising operator relation 
\begin{align}
&
\frac{1}{w(x ;q^{1/2}{\bf t})}\D w(x;{\bf t})  p_n(x, {\bf t}) \label{eqrais} \\ 
& = \frac{2q^{-n/2}(1-q^2/t_1t_2t_3t_4)(1-q^3/t_1t_2t_3t_4)}{(1-q)\prod_{1 \le j < k \le 4}(1+q/t_jt_k)}p_{n+1}(x, q^{1/2}{\bf t}).
\nonumber
\end{align}
\begin{proof}
It readily follows that 
\begin{align*}
&
\frac{1}{W(x;q^{1/2}{\bf t} } \D( W(x;{\bf t})(-q/t_1z,qz/t_1;q)_k) \\
&=\frac{2t_1q^{1/2}}{(1-q)}
[(1-t_1t_2t_3t_4/q^{k+2})(-q^{1/2}/t_1z,q^{1/2}z/t_1;q)_{k+1}\\
&+t_1t_2t_3t_4/q^{k+2}(1+q^{k+1}/t_1t_2)(1+q^{k+1}/t_1t_3)(1+q^{k+1}/t_1t_4)(-q^{1/2}/t_1z,q^{1/2}z/t_1;q)_k].
\end{align*}
This and \eqref{eqDefASP} establish the desired relation.
\end{proof}
Combining \eqref{eqlowering} and \eqref{eqrais} leads to the  $q$-Sturm-Liouville equation  
\begin{align}
&\frac{1}{w(x,\mathbf{t})}\D [w(x,q^{-1/2}\mathbf{t})\D p_n(x;\,\mathbf{t}|q)]=
\frac{4q^{1-n}(1-q^{n})}{(1-q)^2}\\ 
&\times \frac{(1-q^{4}/t_1t_2t_3t_4)(1-q^{5}/t_1t_2t_3t_4)(1-q^{n+3}/t_1t_2t_3t_4)}{\prod_{1 \le j < k \le 4}(1+q^2/t_jt_k)} 
\; p_n(x;\,\mathbf{t}|q).
\nonumber
\end{align}

By iterating \eqref{eqrais} we derive the Rodrigues type  formula 
\begin{align}
&
\frac{1}{w(x;{\bf t})}\D^n w(x;q^{-n/2}{\bf t})\\
&=q^{-n(n-1)/4}\left(\frac{2}{1-q}\right)^n \frac{(q^4/t_1t_2t_3t_4;q)_{2n}}{\prod_{1 \le j < k \le 4}(-q^2/t_jt_k;q)_n}\;  p_n(x;\,\mathbf{t}|q).
\nonumber
\end{align}

A recursion relation for our polynomials follows from the recurrence relation 
of the Askey-Wilson polynomials \cite{And:Ask:Roy,Gas:Rah,Ismbook}. Indeed we find that 
\begin{equation}
2xp_{n}(x,\,\mathbf{t})=A_{n}p_{n+1}(x,\,\mathbf{t})+B_{n}p_{n}(x,\,\mathbf{t})+C_{n}p_{n-1}(x,\,\mathbf{t}),\quad n\ge0,
\label{eq:finite-recurrence-1}
\end{equation}
and
\begin{equation}
x=\frac{z-1/z}{2},\quad p_{-1}(x,\,\mathbf{t})=0,\quad p_{0}(x,\,\mathbf{t})=1,\label{eq:finite-recurrence-2}
\end{equation}
where
\begin{equation}
A_{n}=\frac{1-q^{n+3}/t_{1}t_{2}t_{3}t_{4}}{\left(1-q^{2n+3}/t_{1}t_{2}t_{3}t_{4}\right)\left(1-q^{2n+4}/t_{1}t_{2}t_{3}t_{4}\right)},\label{eq:finite-recurrence-3}
\end{equation}
\begin{align}
& C_{n}=-\frac{\left(1-q^{n}\right)\prod_{1 \le j < k \le 4}
\left(1+q^{n+1}/t_{j}t_{k}\right)}{\left(1-q^{2n+2}/t_{1}t_{2}t_{3}t_{4}\right)\left(1-q^{2n+3}/t_{1}t_{2}t_{3}t_{4}\right)}\label{eq:finite-recurrence-4}, 
\end{align}
and
\bea
\label{eq:finite-recurrence-5}
\bg
  B_{n}=\frac{t_{1}}{q}-\frac{q}{t_{1}}-\frac{t_{1}}{q}A_{n}
\prod_{j=2}^4\left(1+ q^{n+2}/t_{1}t_{j}\right) 
- \frac{qC_{n}}{t_{1} \prod_{j=2}^4\left(1+q^{n+1}/t_{1}t_{j}\right)}.  
\eg
\eea
Note that $A_n$ and $C_n$ are clearly symmetric in all parameters. What is not clear but is nevertheless true is that $B_n$ is also symmetric in $t_1, t_2, t_3, t_4$.

\section{A Finite Family With Positive Weight Function}
 In this section  we shall treat the case
\begin{eqnarray}
\label{eqcompconj}
	t_1 = \overline{t_2},\quad  t_3 = \overline{t_4}, \quad \Im t_1 \ne 0 \quad \text{and} \quad \Im t_3\ne 0,
\end{eqnarray} 
which implies $w(x,\,\mathbf{t})>0.$
 As we saw in \S 2 the moments $\int_{\mathbb{ R}}x^mw(x)dx, 
 m=0,1, \cdots, N2$ exist when $|t_1t_2t_3t_4|<q^{2N+3}$.  
 
 The Sears transformation. (III.15) in \cite{Gas:Rah} shows that $p_n$ is symmetric in the parameters $t_1, t_2, t_3, t_4$. On the other hand it is clear that $p_n$ is a polynomial in $1/t_2, 1/t_3, 1/t_4$, hence it 
 is a also polynomial in $1/t_1$. This means that $p_n$ is a polynomial in the elementary symmetric functions of $1/t_1, 1/t_2, 1/t_3,1/ t_4$, with  real coefficients.  Therefor $p_n$ is a real polynomial of $x$ when $t_1, t_2, t_3, t_4$ are  chosen as  in \eqref{eqcompconj}.
  
Let $\sigma_j$ be the elementary symmetric functions of $t_1, t_2, t_3, t_4$, that is 
\bea
\bg
\sigma_1 = \sum_{j=1}^4 t_j, \quad  \sigma_2= \sum_{1\le j < k \le 4} t_j t_k, \\
\sigma_3 = \sum_{1\le i<  j < k \le 4} t_1t_j t_k, \quad   \sigma_4 = t_1t_2t_3t_4.
\eg
\eea
It is straightforward to show that the weight function satisfies the 
divided difference equations
\begin{eqnarray}
\bg
\frac{\D w(x;q^{-1/2}{\bf t})}{w(x;{\bf t})}
=\frac{2q(1-\sigma_4/q^4)(1- \sigma_4/q^5)}
{(q-1)\prod_{1 \le j < k \le 4}(1+t_jt_k/q^2)}\\
\times \left[2(1-\sigma_4/q^4)x-  \sigma_1/q - \sigma_3/ q^3\right]
\eg
\end{eqnarray}
and 
	\begin{eqnarray}
\bg
\frac{\A w(x;q^{-1/2}{\bf t})}{w(x;{\bf t})} = 
\frac{q^{1/2}(1-\sigma_4/q^4)(1-\sigma_4/q^5)}
{\prod_{1 \le j < k \le 4}(1+t_jt_k/q^2)}\\ 
\times [(2x^2+1)(1+\sigma_4/q^4)+
x (\sigma_3/q^3-\sigma_1/q)+\sigma_2/q^2].
\eg
\end{eqnarray}

 The following lemma is the analogue of integration by parts, \cite{Ism93}.
 \begin{lem}\label{interchange}
	Let $f$ and $g$ with $\int_0^\infty \breve{f}(z)\breve{g}(q^{1/2}z) \frac{dz}{z}$ and $\int_0^\infty \breve{f}(z)\breve{g}
 (q^{-1/2}z)\frac{dz}{z}$ existed. We have 
	\begin{eqnarray}
	\int_{\R}(\D f)(x)g(x)dx = -\int_{\R}(\D g)(x)f(x)dx.
	\end{eqnarray}
\end{lem}
 
When we apply Theorem 21.6.3 in  \cite{Ismbook},  and use the above lemma, Lemma \ref{interchange},  we establish  the following theorem.

\begin{thm}We assume that $g(x;t_1,t_2,t_3,t_4)$ is a positive function and its moments  $\int_\R x^n g(x;t_1,t_2,t_3,t_4)dx $ exist for some non-negative integers $n$ with $0\le n \le 2N$ when $|t_1t_2t_3t_4|<q^{2N+3}$. If $g(x;\, {\bf t})$ satisfies the following equations
	\begin{eqnarray}
	\bg
	\frac{\D g(x;q^{-1/2}{\bf t})}{g(x;{\bf t})}=
	\frac{2q(1-\sigma_4/q^4)(1- \sigma_4/q^5)}
{(q-1)\prod_{1 \le j < k \le 4}(1+t_jt_k/q^2)}\\
\times \left[2(1-\sigma_4/q^4)x-  \sigma_1/q - \sigma_3/ q^3\right]
 \\
\frac{\A g(x;q^{-1/2}{\bf t})}{g(x;{\bf t})}
=
\frac{q^{1/2}(1-\sigma_4/q^4)(1-\sigma_4/q^5)}
{\prod_{1 \le j < k \le 4}(1+t_jt_k/q^2)}\\  
\times [(2x^2+1)(1+\sigma_4/q^4)+
x (\sigma_3/q^3-\sigma_1/q)+\sigma_2/q^2].
\eg
\end{eqnarray}
then $g(x; {\bf t})$ is a weight function for polynomials 
$\{p_n(x;t_1,t_2,t_3,t_4 \mid q)\}_{n=0}^N$.
\end{thm}

\section{Asymptotics}
 The Askey-Wilson polynomials also have an   $_{8}W_{7}$ representation \cite{Gas:Rah}; which, in view of \eqref{eq2.10},  yields the representation 
\begin{align}
& p_{n}(x,\,\mathbf{t})=\frac{\left(-q^{2}/t_{1}t_{2},\,-q^{2}/t_{1}t_{3},\,-q^{2}/t_{2}t_{3},\,-q/t_{4}z;\,q\right)_{n}}{\left(-q^{3}z/t_{1}t_{2}t_{3};\,q\right)_{n}}z^{n}\label{eq:8w7-1}\\
& \times{}_{8}W_{7}\left(-q^{2}z/t_{1}t_{2}t_{3};\,qz/t_{1},\,qz/t_{2},\,qz/t_{3},\,q^{n+3}/t_{1}t_{2}t_{3}t_{4},\,q^{-n};\,q,\,t_{4}/z\right)\nonumber 
\end{align}
and we also have the representation 
\begin{align}
p_{n}(x,\,\mathbf{t}) & =\left(-\frac{q^{2}}{t_{2}t_{3}},-\frac{q^{2}}{t_{2}t_{4}},-\frac{q^{2}}{t_{3}t_{4}};q\right)_{n}\left(Q_{n}\left(-\frac{1}{z},\,\mathbf{t}\right)+Q_{n}\left(z,\,\mathbf{t}\right)\right),\label{eq:8w7-2}
\end{align}
where
\bea
\label{eq:8w7-3}
\bg
  Q_{n}\left(w,\,\mathbf{t}\right)=\frac{\left(-\frac{wq^{n+3}}{t_{1}t_{2}t_{3}},-\,\frac{wq^{n+3}}{t_{2}t_{3}t_{4}},\frac{wq^{n+2}}{t_{2}},\frac{wq^{n+2}}{t_{3}},\,-\frac{q}{wt_{1}},\,-\frac{q}{wt_{2}},\,-\frac{q}{wt_{3}},\,-\frac{q}{wt_{4}};q\right)_{\infty}}{\left(-\frac{q^{2}}{t_{2}t_{3}},-\frac{q^{2}}{t_{2}t_{4}},-\frac{q^{2}}{t_{3}t_{4}},-\frac{q^{n+2}}{t_{1}t_{2}},-\frac{q^{n+2}}{t_{1}t_{3}},\,q^{n+1},\,\frac{w^{2}q^{n+3}}{t_{2}t_{3}},-\frac{1}{w^{2}};q\right)_{\infty}}\\
  \times w^{n}{}_{8}W_{7}\left(\frac{w^{2}q^{n+2}}{t_{2}t_{3}};\,-\frac{q^{n+2}}{t_{2}t_{3}},\frac{qw}{t_{2}},\frac{qw}{t_{3}},\,-t_{1}w,\,-t_{4}w;\,q,\,-\frac{q^{n+2}}{t_{1}t_{4}}\right).
  \eg
  \eea

For each fixed $z\in\mathbb{C}$ the form \eqref{eq:8w7-3} leads to
\begin{equation}
Q_{n}\left(w,\,\mathbf{t}\right)\approx\frac{w^{n}B\left(w^{-1}\right)}{\left(-\frac{q^{2}}{t_{2}t_{3}},-\frac{q^{2}}{t_{2}t_{4}},-\frac{q^{2}}{t_{3}t_{4}};q\right)_{\infty}},\quad n\to\infty,   
\label{eq:8w7-4}
\end{equation}
where
\begin{equation}
B\left(w\right)=\frac{\left(-\frac{qw}{t_{1}},-\frac{qw}{t_{2}},-\frac{qw}{t_{3}},-\frac{qw}{t_{4}};q\right)_{\infty}}{\left(-w^{2};q\right)_{\infty}}.\label{eq:8w7-5}
\end{equation}

Let $z_1, z_2$ be the roots of $2x = z-1/z$, with $|z_1| \le |z_2|$
It is easy to see that $|z_1| = |z_2|$ if and only if $x$ is purely imaginary and $ix \in [-1,1]$. Moreover $z_1 = z_2$ if ad only 
if $x = \pm i$. It is clear that $z_1z_2 =-1$. Therefore 
\bea
p_{n}(x,\,\mathbf{t}) = z_2^n \frac{\left(qz_1/t_{1}, qz_1/t_{2}, 
qz_1/ t_{3}, qz_1/t_{4};q\right)_{\infty}}{\left(-z_1^{2};q\right)_{\infty}}\left[1+ o(1)\right], 
\eea
if $ix \notin [-1, 1]$. If $ix \in (-1, 1)$ we let $z_1= ie^{i\t}, z_2 = 
 ie^{-i\t}$, with $\t \in (0, \pi)$.   Then 
\bea
\label{eqasyOsc}
p_{n}(x,\,\mathbf{t}) = 2 C(x) \cos(n\t + \f-n\pi/2)[1+o(1)],   
\eea
where $C(x) \ge 0$, and 
 \bea 
 C(x) e^{i\f} =  \frac{\left(qi e^{-i\t}/t_{1}, qi e^{-i\t}/t_{2}, 
qi e^{-i\t}/ t_{3}, qi e^{-i\t}/t_{4};q\right)_{\infty}}{\left( e^{-2i\t};q\right)_{\infty}}. 
 \eea
We note that the above asymptotic formulas can also be 
derived  from the generating function \eqref{eqgfnewpol} using 
Darboux's asymptotic method, \cite{Olv}. 

When $n$ is large we know that the polynomials are no longer 
orthogonal but \eqref{eqasyOsc} indicates that the polynomials 
have their zeros dense in the segment connecting $\pm i$.

We now derive the large $N$ asymptotics of $p_N$, when $t_1$ and $t_2$ are of the form $t_1q^N$ and $t_2q^N$, respectively. 

First, we recall the transformation \cite{Gas:Rah} 
\begin{eqnarray}
\bg
	{ }_{4} \phi_{3}\left(\left.\begin{array}{c}
	q^{-n}, a, b, c \\
	d, e, f
	\end{array} \right| q, q\right) \\
	=\left(\frac{bc}{ d}\right)^{n} \frac{(d e / b c, d f / b c ; q)_{n}}{(e, f ; q)_{n}}{ }_{4} \phi_{3}\left(\left.\begin{array}{c}
	q^{-n}, a, d / b, d / c \\
	d, d e / b c, d f / b c
	\end{array} \right| q, q\right).
	\eg
\end{eqnarray}
Let $a=qz/t_1, b=-q/t_1z, c=q^{n+3}/t_1t_2t_3t_4, d=-q^2/t_1t_3, e=-q^2/t_1t_2$ and $f=-q^2/t_1t_4$, and then we have
\begin{eqnarray}
\bg
{}_{4}\phi_{3}\begin{pmatrix}\begin{array}{c}
q^{-n},\,q^{n+3}/t_{1}t_{2}t_{3}t_{4},-q/t_{1}z,\,qz/t_{1}\\
-q^{2}/t_{1}t_{2},-q^{2}/t_{1}t_{3},-q^{2}/t_{1}t_{4}
\end{array} & \bigg|q,\,q\end{pmatrix}	\\= (q^{n+2}/t_1t_2t_4z)^n \frac{(-q^{-n}t_4z,-q^{-n}t_2z;q)_n}{(-q^2/t_1t_2,-q^2/t_1t_4;q)_n}{}_{4}\phi_{3}\begin{pmatrix}\begin{array}{c}
q^{-n},\,-t_2t_4q^{-(n+1)},\,qz/t_{3},\,qz/t_{1}\\
-q^{2}/t_{1}t_{3},-q^{-n}zt_{4},-q^{-n}zt_{2}
\end{array} & \bigg|q,\,q\end{pmatrix}
	\eg
\end{eqnarray}
Next, we define
\begin{eqnarray}
\bg
	v_n = \frac{(-t_1t_3/q, -q/zt_2;q)_n}{(q,t_3/z;q)_n}
	{}_{4}\phi_{3}\begin{pmatrix}\begin{array}{c}
	q^{-n},\,-t_2t_4/q,\,q^{1-n}z/t_{3},\,qz/t_{1}\\
	-q^{2-n}/t_{1}t_{3},-zt_{4},-q^{-n}zt_{2}
	\end{array} & \bigg|q,\,q\end{pmatrix}
	\\=\sum_{k=0}^{n}\frac{(qz/t_1,-t_2t_4/q;q)_k (-t_1t_3/q,-q/zt_2;q)_{n-k}}{(q,-zt_4;q)_k (q,t_3/z;q)_{n-k}}(t_1/t_2)^k.
\eg
\end{eqnarray} 
Therefore, we have 
\begin{eqnarray}
\bg
\Sum v_n t^n \\
= 	{}_{2}\phi_{1} \left(\left. \begin{array}{c}
	qz/t_1,\,-t_2t_4/q\\
	-zt_4
	\end{array}  \right| q,\,\frac{t_1t}{t_2}\right)
	  {}_{2}\phi_{1} \left(\left.\begin{array}{c}
	-q/zt_2,\,-t_1t_3/q\\
	t_3/z
	\end{array} \right|q,\,t \right).
	\eg
\end{eqnarray}
We consider transform 
\begin{eqnarray}
	{ }_{2} \phi_{1}(a, b ; c ; q, z)=\frac{(b, a z ; q)_{\infty}}{(c, z ; q)_{\infty}}{ }_{2} \phi_{1}(c / b, z ; a z ; q, b)
\end{eqnarray}
which leads to
\begin{eqnarray}
\bg
	\Sum v_n t^n = \frac{(-t_1t_3/q,-t_2t_4/q, qzt/t_2,-tq/zt_2;q)_\infty}{(t_3/z,-zt_4,t,tt_1/t_2;q)_\infty} {}_{2}\phi_{1}\begin{pmatrix}\begin{array}{c}
	qz/t_2,\,tt_1/t_2\\
	tqz/t_2
	\end{array} & \bigg|q,\,-\frac{t_2t_4}{q}\end{pmatrix} \\
	\times 
	{}_{2}\phi_{1}\begin{pmatrix}\begin{array}{c}
	qz/t_1,\,t\\
	-tq/zt_2
	\end{array} & \bigg|q,\,-\frac{t_1t_3}{q}\end{pmatrix}
	\eg
\end{eqnarray}
If we apply Darboux's method, we get 
\begin{eqnarray}
\bg
	\lim_{n \rightarrow \infty } v_n = \frac{(-t_1t_3/q,-t_2t_4/q ;q)_\infty}{(t_3/z,-zt_4;q)_\infty}\Big[\frac{(qz/t_2,-q/zt_2;q)_\infty}{(q,t_1/t_2;q)_\infty}{}_{2}\phi_{1}\begin{pmatrix}\begin{array}{c}
	qz/t_2,\,t_1/t_2\\
	qz/t_2
	\end{array} & \bigg|q,\,-\frac{t_2t_4}{q}\end{pmatrix}\\
	+ \left(\frac{t_1}{t_2}\right)^n \frac{(qz/t_1,-q/zt_1;q)_\infty}{(q,t_2/t_1;q)_\infty}	{}_{2}\phi_{1}\begin{pmatrix}\begin{array}{c}
	qz/t_1,\,t_2/t_1\\
	-q/t_1z
	\end{array} & \bigg|q,\,-\frac{t_1t_3}{q}\end{pmatrix} \Big].
	\eg
\end{eqnarray}
Note that
\begin{eqnarray}
\bg
  p_n(x;t_1,t_2,t_3q^n,t_4q^n \mid q)= (q,t_3/z,-t_4z;q)_n q^{-n^2}(q^2/t_1t_2t_4)^n v_n.	
	\eg
\end{eqnarray}
Hence, we can establish the following formula
  \begin{eqnarray}
  \notag
\bg
\frac{ (q^{n-2}t_1t_2t_3t_4)^n}{.(q,-t_1t_3/q,-t_2t_4/q ;q)_\infty} 
p_n(x ; t_1,t_2,q^nt_3, q^n t_4 \mid q)\\
=\left[  t_2^n\frac{(qz/t_2,-q/zt_2;q)_\infty}{(q,t_1/t_2;q)_\infty}{}_{2}\phi_{1}\left(\left. \begin{array}{c}
	qz/t_2,\,t_1/t_2\\
	qz/t_2
	\end{array} \right|q,\,-\frac{t_2t_4}{q} \right) \right.\\ 
\left.+  t_1 ^n \frac{(qz/t_1,-q/zt_1;q)_\infty}{(q,t_2/t_1;q)_\infty}	{}_{2}\phi_{1}\left(\left.\begin{array}{c}
	qz/t_1,\,t_2/t_1\\
	-q/t_1z
	\end{array} \right|q,\,-\frac{t_1t_3}{q}\right)\right][1+o(1)]
	\eg
\end{eqnarray}

\section{An Infinite Family of Orthogonal Polynomials}
 
With the notation  
\bea
\mathbf{t}=(t_1, t_2, t_3)
\eea
 we set 
\bea
\bg
\label{eq:vn-definition}
 V_{n}(x,\,\mathbf{t}|q):=V_{n}(x;\,t_{1},\,t_{2},\,t_{3}|q)\\
 =\left(\frac{t_{1}}{q}\right)^{n} \frac{\left(-q^{2}/t_{1}t_{3};q\right)_{n}}{\left(-q^{2}/t_{2}t_{3};q\right)_{n}}
{}_{3}\phi_{2}\left(\left. \begin{matrix} 
q^{-n},-q/t_{1}z,\,qz/t_{1}\\
-q^{2}/t_{1}t_{3},-q^{2}/t_{1}t_{2}
 \end{matrix} \right| q,\,-\frac{q^{n+2}}{t_{2}t_{3}}\right),
\eg
\eea
and 
\bea
\label{eqneww}
 w(x,\,\mathbf{t}) =\frac{2z \prod_{j=1}^{3}(-t_{j}z,t_{j}/z;q)_{\infty}}{(q,-z^{2},-q/z^{2};q)_{\infty}\log q^{-1}\prod_{1\le j<k\le3}\left(-t_{j}t_{k}/q;q\right)_{\infty}}.   
\eea
It is clear that 
\[
V_{n}(x;\mathbf{t}|q)=\lim_{t_{4}\to0}  
 \frac{p_{n}(x,\,\mathbf{t})}{\left(-q^{2}/t_{1}t_{4},-q^{2}/t_{1}t_{2},-q^{2}/t_{2}t_{3};q\right)_{n}}.
\] 
 When $t_{1},\,t_{2},\,t_{3}>0$ the orthogonality relation 
\eqref{eq:finite-orthogonality}   yields 
\begin{align}
\int_{\mathbb{R}}V_{n}(x;\mathbf{t}\vert q)V_{m}(x;t\vert q)d\mu(x;\mathbf{t}\vert q) & =\frac{(q,-q^2/t_1t_3;q)_n}{(-q^2/t_1t_2,-q^2/t_2t_3;q)_n}\left(\frac{t_1^2}{q^3}\right)^n \delta_{m,n},
\label{eq:22}
\end{align}
where $x$ is as in \eqref{eqxofz}   and 
\begin{align}
& d\mu(x;\mathbf{t}\vert q)=\frac{2z \prod_{j=1}^{3}(-t_{j}z,\,t_{j}/z;q)_{\infty}dx}{(q,-z^{2},-q/z^{2};q)_{\infty}\log q^{-1}\prod_{1\le j<k\le3}\left(-t_{j}t_{k}/q;q\right)_{\infty}}.
\label{eq:23}
\end{align}
 Note that the  measure defined above is a signed measure. 

Next we record the raising and lowering operators for  the polynomials $\{V_n(x;\,\mathbf{t}|q)\}$ by taking the limit as $t_4 \to 0+$  of the corresponding formulas in \S 2. The result is 
\begin{eqnarray}
\label{dqvn}
\D V_n(x;\,\mathbf{t}|q)= \frac{-2q^{(n+3)/2}(1-q^n)V_{n-1}(x;\,q^{-1/2}t_1, q^{-1/2}t_2, q^{-1/2}t_3 |q)}{(1-q)(1+q^2/t_1t_2)(1+q^2/t_2t_3)(t_2t_3)},
\end{eqnarray}       
and 
\begin{eqnarray}\label{rasP}
\bg
\frac{1}{w(x ;q^{1/2}{\bf t})}\D w(x;{\bf t})  V_n(x, {\bf t}) =. 
\frac{2q^{3-n/2}\; V_{n+1}(x, q^{1/2}{\bf t})}{(1-q)t_1^2t_2t_3(1+q/t_1t_3)}
\eg
\end{eqnarray}
Therefore  \eqref{rasP} leads  to the following  $q$-Sturm-Liouville operator equation 
\begin{eqnarray}\label{raisvn}
\bg
\frac{1}{w(x,t_1,t_2,t_3,0)}\D [w(x,q^{-1/2}t_1,q^{-1/2}t_2,q^{-1/2}t_3,0)\D V_n(x;\,\mathbf{t}|q)]\\
= \frac{-4q^{7}}{(1-q)^2t_1^2t_2^2t_3^2} \frac{(1-q^{n})}{\prod_{1 \le j < k \le 3}(1+q^2/t_jt_k)}V_n(x;\,\mathbf{t}|q).
\eg
\end{eqnarray}
By iterating \eqref{rasP} we derive the Rodrigues-type formula 
\begin{eqnarray}
\bg
\frac{1}{w(x;{\bf t})}\D w(x;q^{-n/2}{\bf t})=q^{-n(n-1)/4}\left(\frac{2}{1-q}\right)^n\\ \frac{(q^4/t_1t_2t_3t_4;q)_{2n}}{(-q^2/t_1t_3,-q^2/t_2t_4,-q^2/t_3t_4;q)_n}P_n(x;\,\mathbf{t}|q).
\eg
\end{eqnarray}
Let $t_4 \rightarrow 0$, we obtain another  Rodrigues formula 
\begin{eqnarray}
\label{rodvn}
\bg
\frac{1}{w(x,t_1,t_2,t_3,0)}\D^{n}[w(x,q^{-n/2}t_1,q^{-n/2}t_2,q^{-n/2}t_3,0)]=\\ \left(\frac{2}{1-q}\right)^n\frac{q^{3n^2/4+17n/4}}{t_1^{2n}t_2^nt_3^{n}(-q^2/t_1t_3;q)_n}V_n(x;\mathbf{t}| q).
\eg
\end{eqnarray}

We now prove an orthogonality relation for $\{V_n(x; t_1, t_2, t_3)\}$. 
To prove the orthogonality relations, we need a lemma to compute the total mass at first.  
\begin{lem}\label{totmass}
	$$\int_\R (-qt_1z,qt_1/z,-qt_2z,qt_2/z,-qt_3z,qt_3/z;q)_\infty d\mu_{\alpha}(x)=(-qt_1t_2,-qt_1t_3,-qt_2t_3;q)_\infty$$
	where $d\mu_{\alpha}(x)$ is as in \eqref{mass}.
\end{lem}
 This is a special case of Theorem 3.5 in \cite{Ism2020}
 
 \begin{thm}\label{thmA} The polynomials $\{V_n(x, \mathbf{t})|q\}$ satisfy the orthogonality relation 
\begin{equation}
\label{orthvmoment}
\begin{aligned}
&
\int_\R V_n(x;\,\mathbf{t}|q)V_m(x;\,\mathbf{t}|q)(-qt_1z,qt_1/z,-qt_2z,qt_2/z,-qt_3z,qt_3/z;q)_\infty d\mu_{\alpha}(x)\\
&=\frac{(q,-q^2/t_1t_3;q)_n}{(-q^2/t_1t_2,-q^2/t_2t_3;q)_n}(-qt_1t_2,-qt_1t_3,-qt_2t_3;q)_\infty\left(\frac{t_1^2}{q^3}\right)^n\delta_{m,n}
\end{aligned}
\end{equation}
 where $d\mu_{\alpha}(x)$ is defined by \eqref{mass}.
\end{thm}

\begin{proof}
From the discussion before the proof of Theorem ?? we conclude that 
 \bea
 \notag
 w_A(x_n(\a))= \frac{-\log q \; m_n(\a)}{z_n(\a)+ 1/z_n(\a)}. 
 \eea

Consider  the inner products
\begin{equation}
\langle f, g\rangle_{\alpha}=\sum_{n=-\infty}^{\infty} \breve{f}\left(z_{n}(\alpha)\right) \breve{g}\left(z_{n}(\alpha)\right)\left(q^{-1 / 2}-q^{1 / 2}\right)\left(z_{n}(\alpha)+1 / z_{n}(\alpha)\right) / 2.
\end{equation}
We find $z_{n+1}(\alpha)=q^{-1}z_n(\alpha)$. Therefore 	
\begin{align*}
\left\langle\mathcal{D}_{q} f, g\right\rangle_{\alpha} &=-\sum_{n=-\infty}^{\infty} \breve{g}\left(z_{n}(\alpha)\right)\left[\check{f}\left(q^{1 / 2} z_{n}(\alpha)\right)-\breve{f}\left(q^{-1 / 2} z_{n}(\alpha)\right)\right] \\ &=\sum_{n=-\infty}^{\infty} \breve{g}\left(z_{n}(\alpha)\right) \breve{f}\left(q^{-1 / 2} z_{n}(\alpha)\right)-\sum_{n=-\infty}^{\infty} \breve{g}\left(z_{n+1}(\alpha)\right) \breve{f}\left(q^{1 / 2} z_{n+1}(\alpha)\right) \\ &=\sum_{n=-\infty}^{\infty} \breve{g}\left(z_{n}(\alpha)\right) \breve{f}\left(z_{n}\left(\alpha q^{1 / 2}\right)\right)-\sum_{n=-\infty}^{\infty} \breve{g}\left(q^{-1 / 2} z_{n}\left(\alpha q^{1 / 2}\right)\right) \breve{f}\left(z_{n}\left(\alpha q^{1 / 2}\right)\right) \\
&=\sum_{n=-\infty}^{\infty} \check{f}\left(z_{n}\left(\alpha q^{1 / 2}\right)\right)\left[\breve{g}\left(q^{1 / 2} z_{n}\left(\alpha q^{1 / 2}\right)\right)-\breve{g}\left(q^{-1 / 2} z_{n}(\alpha)\right)\right]\\
&=-\left\langle f, \mathcal{D}_{q} g\right\rangle_{\alpha q^{1 / 2}}
\end{align*}
Let 
\bea
\notag 
\lambda_n(\mathbf{t})=\frac{-4q^{7}}{(1-q)^2t_1^2t_2^2t_3^2}  \frac{(1-q^{n})}{\prod_{1 \le j < k \le 3}(1+q^2/t_jt_k)}.
\eea
 We  apply  \eqref{raisvn} and   use \eqref{askfp2},  to  find that $\lambda_n$ times the left-side of \eqref{orthvmoment} is equal to 
\begin{equation}
\notag
	\begin{aligned}
	&\int_{\mathbb{R}} \frac{(-q t_1 z, q t_1 / z,-q t_2 z, q t_2 / z ,-q t_3 z, q t_3 / z ; q)_{\infty}}{w(x ; t_1, t_2, t_3,0)} V_{m}(x ; t_1, t_2, t_3)  \\
	& \times \mathcal{D}_{q}\left[w\left(x ; t_1 q^{-1 / 2}, t_2 q^{-1 / 2}, t_3 q^{-1 / 2}, 0 \right) \mathcal{D}_{q} V_{n}(x ;t_1, t_2, t_3)\right] d \mu_{\alpha}(x) \\
	=& \int_{\mathbb{R}} \frac{V_{m}(x ;t_1, t_2, t_3)}{2 w_{A}(x)} \mathcal{D}_{q}\left[w\left(x ; t_1 q^{-1 / 2}, t_2 q^{-1 / 2}, t_3 q^{-1 / 2},0\right) \mathcal{D}_{q} V_{n}(x ; t_1, t_2, t_3)\right] d \mu_{\alpha}(x) \\
	=& \sum_{j=-\infty}^{\infty} V_{m}\left(x_{j}(\alpha) ;  t_1, t_2, t_3\right) \frac{z_{j}(\alpha)+1 / z_{j}(\alpha)}{2} \\
	& \times \mathcal{D}_{q}\left[w\left(x ; t_1 q^{-1 / 2}, t_2 q^{-1 / 2}, t_3 q^{-1 / 2},0\right) \mathcal{D}_{q} V_{n}(x ; t_1, t_2, t_3)\right]_{x=x_{j}(\alpha)}.
	\end{aligned}
\end{equation}
Hence,
\begin{equation}
\notag
	\begin{aligned}
	&\lambda_n \int_{\mathbb{R}} V_{m}(x ; t_1 ,t_2,  t_3) V_{n}(x ; t_1 ,t_2,  t_3)(-q t_1 z, q t_1 / z,-q t_2 z, q t_2 / z ,-q t_3 z, q t_3 / z ; q)_{\infty} d \mu_{\alpha} \\
	=&\frac{1}{q^{1 / 2}-q^{-1 / 2}}\left\langle\mathcal{D}_{q}\left[w\left(x ; t_1 q^{-1 / 2}, t_2 q^{-1 / 2}, t_3 q^{-1 / 2},0\right) \mathcal{D}_{q} V_{n}(x ;t_1 ,t_2,  t_3)\right], V_{m}(x ;t_1 ,t_2,  t_3)\right\rangle_{\alpha} \\
=&-\frac{1}{q^{1 / 2}-q^{-1 / 2}}\left\langle w\left(x ; t_1 q^{-1 / 2}, t_2 q^{-1 / 2}, t_3 q^{-1 / 2},0\right) \mathcal{D}_{q} V_{n}(x ; t_1, t_2, t_3)\right], \mathcal{D}_{q} V_{m}(x ; t_1, t_2, t_3)\rangle_{\alpha q^{1 / 2}}.
	\end{aligned}
\end{equation}
Note that the above formula is symmetric in $m$ and $n$ and that 
$\lambda_n$ is strictly monotonous in n. Hence the integral 
\eqref{orthvmoment} will vanish if $m\neq n$.
$\int_\R d\mu(x)$ has been computed  in  Lemma~\ref{totmass}.
To calculate the $\int_\R p_n^2(x)d\mu(x)$, we apply the well-known 
 fact that if a sequence of orthogonal polynomial have the following the recurrence relations 
$$p_{n+1}(x)=[A_nx+B_n]p_n(x)-C_np_{n-1}(x),$$
then their orthogonality relation is 
$$\int_{\mathbb{R}} p_{n}(x) p_{m}(x) d \mu(x)=\frac{\delta_{m, n} A_{0}}{A_{n}} C_{1} \cdots C_{n} \int_{\mathbb{R}} d \mu(x).$$
We apply this fact and use the recurrence relation \eqref{eq:vn-recurrence-3}, to establish \eqref{orthvmoment} up to the evualtion of the integral $\int_\R d\mu(x)$. This is the complete proof.
\end{proof}

 For completeness we record the three term recurrence relation of the polynomials $\{V_{n}(x;\,\mathbf{t}\vert q)\}$ which follows as 
 the  limiting case $t_4\to 0+$ of \eqref{eq:finite-recurrence-1}. 
The result is
\begin{equation}
2xV_{n}(x;\,\mathbf{t}\vert q)=a_{n}V_{n+1}(x;\,\mathbf{t}\vert q)+b_{n}V_{n}(x;\,\mathbf{t}\vert q)+c_{n}V_{n-1}(x;\,\mathbf{t}\vert q),\quad n\ge0,\label{eq:vn-recurrence-3}
\end{equation}
with 
\begin{equation}
x=\frac{z-1/z}{2},\quad V_{-1}(x;\,\mathbf{t}\vert q)=0,\quad V_{0}(x;\,\mathbf{t}\vert q)=1,\label{eq:vn-recurrence-4}
\end{equation}
where, 
\bea
\label{eq:vn-recurrence-5}
\bg
a_{n}  =-\frac{t_{2}t_{3}}{q^{2n+2}}\left(1+q^{2+n}/t_{1}t_{2}\right)\left(1+q^{2+n}/t_{2}t_{3}\right),\\
b_{n}  =\frac{t_1+t_2+t_3}{q^{n+1}}+\frac{t_1t_2t_3}{q^{2n+3}}(1+q-q^{n+1})
, \\
c_{n}  =-\frac{t_{1}^{2}t_{2}t_{3}}{q^{2n+3}}\left(1-q^{n}\right)\left(1+q^{n+1}/t_{1}t_{3}\right).\label{eq:vn-recurrence-7}
\eg
\eea
The $V_n$'s are not symmetric but the following renormalization is symmetric in the parameters $t_1, t_2, t_3$, 
The symmtric form is 
\bea
\bg
\label{eq:vnt-definition}
\tilde{V}_{n}(x,\,\mathbf{t}|q):=\tilde{V}_{n}(x;\,t_{1},\,t_{2},\,t_{3}|q)\\
= \frac{1}{\left(-q^{2}/t_{2}t_{3};q\right)_{n}}\;
{}_{3}\phi_{2}\left(\left. \begin{matrix}
q^{-n},-q/t_{1}z,\,qz/t_{1}\\
-q^{2}/t_{1}t_{3},-q^{2}/t_{1}t_{2}
\end{matrix} \right| q,\,-\frac{q^{n+2}}{t_{2}t_{3}}\right).
\eg
\eea

We now wish to study the connection coefficients for our polynomials.. 
One can use the Rodrigues-type  formula to derive  the desired connection  relation by following  the proof of the corresponding result for the \aw polynomials given in \cite{Ism:Zha1}, see also \cite{Ismbook}. 
\begin{thm}
We have 
\begin{equation}
V_{n}(x,\,\mathbf{s})=\sum_{k=0}^{n}e_{k,n}\left(\mathbf{s},\,\mathbf{t}\right)V_{k}(x,\,\mathbf{t}),\label{eq:infinite-connection-5}
\end{equation}
where 
\begin{equation}
\mathbf{t}=t_{1},\ t_{2},\ t_{3}\label{eq:infinite-connection-6}
\end{equation}
and
\begin{equation}
\mathbf{s}=s_{1},\ s_{2},\ s_{3}\label{eq:infinite-connection-7}, \quad \quad s_3=t_3
\end{equation}
with
\begin{align}
& e_{k,n}\left(\mathbf{s},\,\mathbf{t}\right)=\frac{\left(-q^{2}/s_{1}t_{3},\,q;q\right)_{n}\left(-q^2/t_2t_3,-q^2/t_1t_2;q\right)_{k}}{\left(-q^2/s_1s_2;q\right)_n\left(-q^{2}/s_{1}t_{3},-q^{2}/s_{2}t_{3},\,q;q\right)_{k}\left(q;q\right)_{n-k}}q^{k-n}s_1^n\left(\frac{t_2}{s_1s_2}\right)^k\label{eq:infinite-connection-4}\\
& \times{}_{3}\phi_{2}\begin{pmatrix}\begin{array}{c}
q^{k-n},\,-q^{k+2}/t_{1}t_{3},-q^{k+2}/t_{2}t_{3}\\
-q^{k+2}/s_{1}t_{3},-q^{k+2}/s_{2}t_{3},
\end{array} & \bigg|q,\  \frac{q^{n-k-1}t_1t_2}{s_1s_2}\end{pmatrix}.\nonumber 
\end{align}
\end{thm}

For completeness we record a generating function for the polynomials $\{V_n\}$.  We let  $t_{4}\to0$ in the generating 
function  \eqref{eqgfnewpol} and conclude that

\bea
\bg
{}_{2}\phi_{1}\begin{pmatrix}\begin{array}{c}
qz/t_{1},\:qz/t_{3}\\
-q^{2}/t_{1}t_{3}
\end{array} & \bigg|q,\,-\frac{t}{z}\end{pmatrix}\frac{\left(-t/z;q\right)_{\infty}}{\left(tt_{3}/q;q\right)_{\infty}} \qquad  \\
\qquad\qquad
 =\sum_{n=0}^{\infty}\frac{\left(-q^{2}/t_{2}t_{3};q\right)_{n}}{\left(q;q\right)_{n}}V_{n}(x;\,\mathbf{t}\vert q)\left(\frac{tt_{3}}{t_{1}}\right)^{n},
 \eg
 \label{eq:vn-gf-1}
\eea
Using the symmetry of $\tilde{V}_n$, in its parameters, see 
\eqref{eq:vnt-definition} we can write additional equivalent 
generating functions. 
 
 We conclude this section with a theorem which allows us to generate many additional weight functions for the polynomials 
 $\{V_n(x; {\bf t})\}$.    

The product rule for $\D$ is 
\begin{eqnarray}\label{dfg}
(\D fg)(x)= (\A f)(x) (\D g)(x)+(\A g)(x) (\D f)(x).
\end{eqnarray}

As was pointed out in \cite{Ism2020} and under the assumptions in 
Lemma~\ref{interchange}  the eigenfunctions of 
\begin{eqnarray}\label{q-Strum}
\frac{1}{v(x)}\D[v(x)\D y]= \lambda y.
\end{eqnarray}
corresponding to distinct eigenvalues are orthogonal with respect to 
$v$ on $\R$, if $v(x) > 0$ on $\R$.  The proof also uses \eqref{dfg}. 
 
\begin{thm}We assume that $f(x;t_1,t_2,t_3)$ is a positive function and its moments $\int_\R x^n f(x;t_1,t_2,t_3)dx $ exist for all non-negative integers n. If $f$ satisfies the following equations
	\begin{equation}
	\begin{aligned}
		&
	\frac{\D f(x;q^{-1/2}{\bf t})}{f(x; {\bf t})}\\
	&
	=\frac{2q(2x-t_1/q-t_2/q-t_3/q-t_1t_2t_3/q^3)}{(q-1)\prod_{1 \le j < k \le 3}(1+t_jt_k/q^2)} \; 
	\frac{\A f(x;q^{-1/2}{\bf t})}{f(x;\, {\bf t})}\\ 
	&
	=\frac{q^{1/2}}{\prod_{1 \le j < k \le 3}(1+t_jt_k/q^2)}\\
	&
	\times [2x^2+1+x(t_1t_2t_3/q^3-t_1/q-t_2/q-t_3/q)+(t_1t_2+t_1t_3+t_2t_3)/q^2].
	\end{aligned}
	\end{equation}
	then $f$ is weight function for polynomials $\{V_n(x; {\bf t} \mid q)\}_{n=0}^\infty$.
\end{thm}

The proof follows from Lemma \ref{interchange} and  \eqref{dfg}. The details are identical to the the proof of the corresponding result in  \cite{Ism2020}.

\section{Pointwise and Plancherel-Rotach Asymptotics} 
A theorem  of Ismail and Li \cite{Ism:Li}, which is also stated as 
Theorem 7.2.7 in  
 \cite{Ismbook}, implies that the largest zero of $V_n$ is $\O(q^{-2n})$.

Since
\begin{align*}
& \frac{q^{n}V_{n}(x;\,t_{1},\,t_{2},\,t_{3}|q)}{t_{1}^{n}}\frac{\left(-q^{2}/t_{2}t_{3};q\right)_{n}}{\left(-q^{2}/t_{1}t_{3};q\right)_{n}}=
{}_{3}\phi_{2} \left(\left. \begin{array}{c}
q^{-n},-q/t_{1}z,\,qz/t_{1}\\
-q^{2}/t_{1}t_{3},-q^{2}/t_{1}t_{2}
\end{array} \right| q,\,-\frac{q^{n+2}}{t_{2}t_{3}}\right)\\
& =\sum_{k=0}^{n}\frac{\left(-q/t_{1}z,\,qz/t_{1};q\right)_{k}}{\left(q,-q^{2}/t_{1}t_{3},-q^{2}/t_{1}t_{2};q\right)_{k}}\left(\frac{q^{2}}{t_{2}t_{3}}\right)^{k}\left(q^{-n};q\right)_{k}\left(-q^{n}\right)^{k}. 
\end{align*}
Write $(q^{-n};q)_k$ as $ (-1)^kq^{\binom{k}{2}} (q;q)_n/(q;q)_{n-k}$, 
then apply Tannery's theorem (the discrete analogue of the Lebesgue Dominated Convergence Theorem) to  get,
\begin{equation}
\begin{aligned}
&
\lim_{n\to\infty}\frac{q^{n}V_{n}(x;\,t_{1},\,t_{2},\,t_{3}|q)}{t_{1}^{n}}
\qquad      \qquad         \\
&
=\frac{\left(-q^{2}/t_{1}t_{3};q\right)_{\infty}}
{\left(-q^{2}/t_{2}t_{3};q\right)_{\infty}}\; {}_{2}\phi_{2} 
\left(\left.
\begin{array}{c}
-q/t_{1}z,\,qz/t_{1}\\
-q^{2}/t_{1}t_{3},-q^{2}/t_{1}t_{2}
\end{array} \right|  q,\,-\frac{q^{2}}{t_{2}t_{3}}\right).
\end{aligned}
\label{eq:asymptotics-2}
\end{equation}
The limit is uniform on compact subsets of the complex plane. 
In particular, it implies that the entire function,

\[
_{2}\phi_{2}\left(\left. \begin{array}{c}
-ae^{\xi},\,ae^{-\xi}\\
-ab,-ac
\end{array}    \right|q,\,-bc\right),\quad a,b,c>0
\]
has only real zeros. 

In order to develop  the Plancherel--Rotach asymptotics we set 
\bea
z = q^{-cn}/s, \quad x = (z-1/z)/2.
\eea 

Write the ${}_3\phi_2$ in the definition of $V_n$ in 
 \eqref{eq:vn-definition} as $\sum_{k=0}^n$. We shall use the identities
\bea
(q^{-n};q)_n = \frac{(q;q)_n}{(q;q)_{n-k}} 
(-1)^k q^{\binom{k}{2} -nk}, 
\qquad (A;q)_k = (q^{1-k}/A;q)_k A^kq^{\binom{k}{2}}.
\eea 
 
\begin{thm} 
Let $z_n(s)=q^{-cn}/s$, then $x_n(s)=(q^{-cn}/s-q^{cn}s)/2$. 

\noindent  $\textup{(a)} $ If $c=2$, the polynomial $V_n$ 
has the  soft edge limiting behavior   
\begin{eqnarray}
\bg
\lim_{n\to\infty} (-st_2t_3)^nq^{n^2-n} V_n(x_n(s);t_1,t_2,t_3)\\
=\frac{1}{(-q^2/t_1t_2,-q^2/t_2t_3;q)_\infty}\left(\sum_{k=0}^{\infty}\frac{(-1)^kq^{k^2-2k}(st_1t_2t_3)^k}{(q;q)_k}\right) = \frac{A_q(st_1t_2t_3q^{-2})}{(-q^2/t_1t_2,-q^2/t_2t_3;q)_\infty}.    
\eg
\end{eqnarray}
 
\noindent  $\textup{(b)}$ If  $c>2$, then $V_n$ has the 
asymptotic property  
\begin{eqnarray}\label{clarge2}
\bg
\quad 		\lim_{n\to\infty}q^{(c-1)n^2-n}(-st_2t_3)^n 
V_n(x_n(s);t_1,t_2,t_3 \mid q)  
= \frac{1}{(-q^2/t_1t_2,-q^2/t_2t_3;q)_\infty}.
		\eg
	\end{eqnarray}
	
\noindent  $\textup{(c)}$ If $1\le c<2,$ we set $m=\lfloor{\frac{cn}{2}}\rfloor$ and $r=\{\frac{cn}{2}\} = \frac{cn}{2}-m$.  For fixed $r$ and as $n \to \infty$ we have 
	\begin{eqnarray}
\bg
    \lim_{n\to\infty} (-st_2t_3)^mq^{n+c^2n^2/4-2m}t_1^{m-n} V_n(x_n(s);t_1,t_2,t_3)\\=\frac{q^{r^2}(q^2,\frac{q^{3+2r}}{st_1t_2t_3},q^{-(1+2r)}st_1t_2t_3;q^2)_\infty}{(q,-q^2/t_1t_2,-q^2/t_2t_3;q)_\infty}.
    \eg
\end{eqnarray}	
\end{thm}
 
\begin{proof}[Proof of $\textup{(a)}$]
	First we find that
	\begin{equation}\label{phi32}
	\begin{aligned}
		&
	\frac{q^{n}V_{n}(x;\,t_{1},\,t_{2},\,t_{3}|q)}{t_{1}^{n}}\frac{\left(-q^{2}/t_{2}t_{3};q\right)_{n}}{\left(-q^{2}/t_{1}t_{3};q\right)_{n}}\\
	&=\sum_{k=0}^{n}\frac{\left(-q/t_{1}z,\,qz/t_{1};q\right)_{k}}{\left(q,-q^{2}/t_{1}t_{3},-q^{2}/t_{1}t_{2};q\right)_{k}}\left(\frac{q^{2}}{t_{2}t_{3}}\right)^{k}\left(q^{-n};q\right)_{k}\left(-q^{n}\right)^{k}\\
	&
	=\sum_{k=0}^{n}\frac{\left(-q/t_{1}z,\,t_{1}/zq^k;q\right)_{k}}{\left(q,-q^{2}/t_{1}t_{3},-q^{2}/t_{1}t_{2};q\right)_{k}}\left(\frac{q^{2}z}{t_1t_{2}t_{3}}\right)^{k}\frac{(q;q)_{n}}{(q;q)_{n-k}}\left(q\right)^{k^2}(-1)^k.
	\end{aligned}
	\end{equation}
	We replace $z$ by $z_n=q^{-2n}/s$ and conclude that  the 
	right-hand side of the above formula is equal to
	\begin{eqnarray}
	\notag
	q^{-n^2}(q;q)_n\sum_{k=0}^{n}\frac{\left(-q^{1+2n}s/t_{1},\,st_{1}q^{2n-k};q\right)_{k}}{\left(q,-q^{2}/t_{1}t_{3},-q^{2}/t_{1}t_{2};q\right)_{k}}\left(-\frac{q^{2}}{st_1t_{2}t_{3}}\right)^{k}\frac{q^{(n-k)^2}}{(q;q)_{n-k}}.
	\end{eqnarray}
	Replace $k$ by $n-k$ in the above sum to change it to      
	\begin{eqnarray}
	\notag
	q^{-n^2}(q;q)_n\sum_{k=0}^{n}\frac{\left(-q^{1+2n}s/t_{1},\,t_{1}q^{n+k}s;q\right)_{n-k}}{\left(q,-q^{2}/t_{1}t_{3},-q^{2}/t_{1}t_{2};q\right)_{n-k}}\left(-\frac{q^{2}}{st_1t_{2}t_{3}}\right)^{n-k}\frac{q^{k^2}}{(q;q)_{k}}.
	\end{eqnarray}
	As $n\to \infty$,  we use Tannery’s theorem  to establish 
	part (a) of the theorem.
\end{proof}
\begin{proof}[Proof of $\textup{(b)}$]
First we apply \eqref{phi32} and replace $z$ by $q^{-cn}/s$.  The 
result is 
\begin{align*}
&
	\frac{q^{n}V_{n}(x_n(s);\,t_{1},\,t_{2},\,t_{3}|q)}{t_{1}^{n}}\frac{\left(-q^{2}/t_{2}t_{3};q\right)_{n}}{\left(-q^{2}/t_{1}t_{3};q\right)_{n}}\\
&
=\sum_{k=0}^{n}\frac{\left(-q^{1+cn}s/t_{1},\,t_{1}s/q^{k-cn};q\right)_{k}}{\left(q,-q^{2}/t_{1}t_{3},-q^{2}/t_{1}t_{2};q\right)_{k}}\left(\frac{q^{2-cn}}{st_1t_{2}t_{3}}\right)^{k}\frac{(q;q)_{n}}{(q;q)_{n-k}}\left(q\right)^{k^2}(-1)^k
\end{align*}
When we interchange $k$ and  $n-k$, the above expression   becomes 
\begin{align*}
&
\left(\frac{q^{(1-c)n+2}}{-st_1t_2t_3}\right)^n \sum_{k=0}^{n}\frac{\left(-q^{1+cn}s/t_{1},\,t_{1}s/q^{(1-c)n-k};q\right)_{n-k}}{\left(q,-q^{2}/t_{1}t_{3},-q^{2}/t_{1}t_{2};q\right)_{n-k}}  \\
&
\times \left(-st_1t_{2}t_{3}\right)^{k}\frac{(q;q)_{n}}{(q;q)_{k}} \left(q\right)^{k^2-2k+(c-2)nk}.
\end{align*} 
Let $n\to \infty$, and use Tannery’s theorem, establish   \eqref{clarge2}.
\end{proof}
 \begin{proof}[Proof of $\textup{(c)}$]First, we write the sum in \eqref{eq:vn-definition} as $\sum_{k=0}^m + \sum_{k= m+1}^n$. Next, replace $k$ by $m-k$ in the first sum, and replace $k$ by $k+m+1$ in the second sum. This leads to  
	 \begin{align*}
&
	 \frac{q^{n+\frac{c^2n^2}{4}}V_{n}(x_n(s);\,t_{1},\,t_{2},\,t_{3}|q)}{t_{1}^{n}}\frac{\left(-q^{2}/t_{2}t_{3};q\right)_{n}}{\left(-q^{2}/t_{1}t_{3};q\right)_{n}} \\
&=	 \sum_{k=0}^{m}\frac{\left(-q^{1+cn}s/t_{1},\,t_{1}s/q^{m-k-cn};q\right)_{m-k}}{\left(q,-q^{2}/t_{1}t_{3},-q^{2}/t_{1}t_{2};q\right)_{m-k}} \left(\frac{-q^{2}}{st_1t_{2}t_{3}}\right)^{m-k}  \frac{(q;q)_{n}}{(q;q)_{n+k-m}}  q^{(k+r)^2} 
	 \\
&+ 
	\sum_{k=0}^{n-m-1}\frac{\left(-q^{1+cn}s/t_{1},\,t_{1}s/q^{k+m+1-cn};q\right)_{k+m+1}}{\left(q,-q^{2}/t_{1}t_{3},-q^{2}/t_{1}t_{2};q\right)_{k+m+1}}\left(\frac{-q^{2}}{st_1t_{2}t_{3}}\right)^{k+m+1}\frac{(q;q)_{n}}{(q;q)_{n-m-k-1}} q^{(k+1-r)^2}.
	 \end{align*}
	  For fixed $r$, then we apply Tannery's theorem  as $n \to \infty$ and use the Jacobi triple product identity. The result is  
	\begin{align*}
&
\lim_{n\to\infty} (-st_2t_3)^mq^{n+c^2n^2/4-2m}t_1^{m-n} V_n(x_n(s);t_1,t_2,t_3)\\
&=\frac{1}{(q,-q^2/t_1t_2,-q^2/t_2t_3;q)_\infty}\sum_{k=0}^\infty(-1)^k\left[q^{(k+r)^2}\left(\frac{q^2}{st_1t_2t_3}\right)^{-k}-q^{(k+1-r)^2}\left(\frac{q^2}{st_1t_2t_3}\right)^{k+1}\right]\\
&
=\frac{1}{(q,-q^2/t_1t_2,-q^2/t_2t_3;q)_\infty}\sum_{k=-\infty}^\infty(-1)^kq^{(k+r)^2}\left(\frac{q^2}{st_1t_2t_3}\right)^{k}\\
&
=\frac{q^{r^2}(q^2,\frac{q^{3+2r}}{st_1t_2t_3},q^{-(1+2r)}st_1t_2t_3;q^2)_\infty}{(q,-q^2/t_1t_2,-q^2/t_2t_3;q)_\infty}.
\end{align*}
This complete the proof. 
\end{proof}

In the rest of the section we consider cases where at least one parameter depends on the degree $n$ and $n \to \infty$. 
For any $\alpha>0$, let 
\begin{equation}
t_{1}\to t_{1}q^{-n\alpha},\quad x_{1}(n)=\frac{q^{-n\alpha}z-q^{n\alpha}/z}{2}.\label{eq:asymptotics-3}
\end{equation}
Since,
\begin{align*}
& \frac{q^{n^{2}\alpha+n}\left(-q^{2}/t_{2}t_{3};q\right)_{n}}{t_{1}^{n}\left(-q^{2+n\alpha}/t_{1}t_{3};q\right)_{n}}V_{n}\left(x_{1}(n);\,t_{1}q^{-n\alpha},\,t_{2},\,t_{3}\vert q\right)\\
& =_{3}\phi_{2}\begin{pmatrix}\begin{array}{c}
q^{-n},-q^{2n\alpha}/t_{1}z,\,qz/t_{1}\\
-q^{2+n\alpha}/t_{1}t_{3},-q^{2+n\alpha}/t_{1}t_{2}
\end{array} & \bigg|q,\,-\frac{q^{n+2}}{t_{2}t_{3}}\end{pmatrix},
\end{align*}
then,
\begin{equation}
\lim_{n\to\infty}\frac{q^{n^{2}\alpha+n}}{t_{1}^{n}}V_{n}\left(x_{1}(n);\,t_{1}q^{-n\alpha},\,t_{2},\,t_{3}\vert q\right)=\frac{\left(-q^{3}z/t_{1}t_{2}t_{3};q\right)_{\infty}}{\left(-q^{2}/t_{2}t_{3};q\right)_{\infty}^{2}}.\label{eq:asymptotics-4}
\end{equation}

\begin{thm}
	\label{thm:qairy}Let 
	\begin{equation}
	\alpha>\beta>0,\ \gamma,\delta>0,\ \gamma+\delta=\alpha-\beta=1\label{eq:qairy-1}
	\end{equation}
	and 
	\begin{equation}
	z,\,t_{1},\,t_{2},\,t_{3}\in\mathbb{C},\quad z\cdot t_{1}\cdot t_{2}\cdot t_{3}\neq0.\label{eq:qairy-2}
	\end{equation}
	Then there exists a positive number 
	\begin{equation}
	0<\eta<\min\left\{ 1,\,(\alpha+\beta),\,(\beta+\delta),\,(\beta+\gamma)\right\} \label{eq:qairy-3}
	\end{equation}
	such that
	\begin{equation}
	V_{n}\left(x_{2}(n);\,t_{1}q^{-n\beta},\,t_{2}q^{-n\gamma},\,t_{3}q^{-n\delta}\big|q\right)=\frac{t_{1}^{n}}{q^{n^{2}\beta+n}}\left(A_{q}\left(\frac{q^{2}z}{t_{1}t_{2}t_{3}}\right)+\mathcal{O}\left(q^{\eta n}\right)\right)\label{eq:qairy-4}
	\end{equation}
	as $n\to\infty$.
\end{thm}

\begin{proof}
	Let 
	\[
	t_{1}\to t_{1}q^{-n\beta},\ t_{2}\to t_{2}q^{-n\gamma},\ t_{3}\to t_{3}q^{-n\delta},\ x_{2}(n)=\frac{q^{-n\alpha}z-q^{n\alpha}/z}{2}.
	\]
	Then, 
	\begin{align*}
	& V_{n}\left(x_{2}(n);\,t_{1}q^{-n\beta},\,t_{2}q^{-n\gamma},\,t_{3}q^{-n\delta}\big|q\right)\frac{q^{n^{2}\beta+n}\left(-q^{2+n}/t_{2}t_{3};q\right)_{n}}{t_{1}^{n}\left(-q^{2+n(\beta+\delta)}/t_{1}t_{3};q\right)_{n}}\\
	& =_{3}\phi_{2}\begin{pmatrix}\begin{array}{c}
	q^{-n},-q^{1+n(\alpha+\beta)}/t_{1}z,\,q^{1-n(\gamma+\delta)}z/t_{1}\\
	-q^{2+n(\beta+\delta)}/t_{1}t_{3},-q^{2+n(\beta+\gamma)}/t_{1}t_{2}
	\end{array} & \bigg|q,\,-\frac{q^{2+2n}}{t_{2}t_{3}}\end{pmatrix}\\
	& =\sum_{k=0}^{n}\frac{\left(-q^{2}/t_{2}t_{3}\right)^{k}}{\left(q;q\right)_{k}}\frac{\left(q^{-n},\,q^{1-n}z/t_{1};q\right)_{k}q^{2kn}\left(-q^{1+n(\alpha+\beta)}/t_{1}z;q\right)_{k}}{\left(-q^{2+n(\beta+\delta)}/t_{1}t_{3},-q^{2+n(\beta+\gamma)}/t_{1}t_{2};q\right)_{k}}\\
	& =\sum_{k=0}^{n}\frac{q^{k^{2}}}{\left(q;q\right)_{k}}\frac{\left(q;q\right)_{n}}{\left(q;q\right)_{n-k}}\frac{\left(t_{1}/z;q\right)_{n}}{\left(t_{1}/z;q\right)_{n-k}}\frac{\left(-q^{2}z/t_{1}t_{2}t_{3}\right)^{k}\left(-q^{1+n(\alpha+\beta)}/t_{1}z;q\right)_{k}}{\left(-q^{2+n(\beta+\delta)}/t_{1}t_{3},-q^{2+n(\beta+\gamma)}/t_{1}t_{2};q\right)_{k}}\\
	& =S_{1}(n)+S_{2}(n),
	\end{align*}
	where
	\[
	S_{1}(n)=\sum_{k=0}^{\left\lfloor \sqrt{n}\right\rfloor }\frac{q^{k^{2}}}{\left(q;q\right)_{k}}\frac{\left(q;q\right)_{n}}{\left(q;q\right)_{n-k}}\frac{\left(t_{1}/z;q\right)_{n}}{\left(t_{1}/z;q\right)_{n-k}}\frac{\left(-q^{2}z/t_{1}t_{2}t_{3}\right)^{k}\left(-q^{1+n(\alpha+\beta)}/t_{1}z;q\right)_{k}}{\left(-q^{2+n(\beta+\delta)}/t_{1}t_{3},-q^{2+n(\beta+\gamma)}/t_{1}t_{2};q\right)_{k}}
	\]
	and
	\[
	S_{2}(n)=\sum_{k=\left\lfloor \sqrt{n}\right\rfloor +1}^{n}\frac{q^{k^{2}}}{\left(q;q\right)_{k}}\frac{\left(q;q\right)_{n}}{\left(q;q\right)_{n-k}}\frac{\left(t_{1}/z;q\right)_{n}}{\left(t_{1}/z;q\right)_{n-k}}\frac{\left(-q^{2}z/t_{1}t_{2}t_{3}\right)^{k}\left(-q^{1+n(\alpha+\beta)}/t_{1}z;q\right)_{k}}{\left(-q^{2+n(\beta+\delta)}/t_{1}t_{3},-q^{2+n(\beta+\gamma)}/t_{1}t_{2};q\right)_{k}}.
	\]
	Since
	\[
	\frac{\left(q;q\right)_{n}}{\left(q;q\right)_{n-k}}\frac{\left(t_{1}/z;q\right)_{n}}{\left(t_{1}/z;q\right)_{n-k}}=1+\mathcal{O}\left(q^{n-\sqrt{n}}\right),\quad n\to\infty
	\]
	and
	\begin{align*}
	& \frac{\left(-q^{1+n(\alpha+\beta)}/t_{1}z;q\right)_{k}}{\left(-q^{2+n(\beta+\delta)}/t_{1}t_{3},-q^{2+n(\beta+\gamma)}/t_{1}t_{2}\right)_{k}}=\frac{\left(-q^{1+n(\alpha+\beta)}/t_{1}z;q\right)_{\infty}}{\left(-q^{1+n(\alpha+\beta)+k}/t_{1}z;q\right)_{\infty}}\\
	& \times\frac{\left(-q^{2+n(\beta+\delta)+k}/t_{1}t_{3};q\right)_{\infty}}{\left(-q^{2+n(\beta+\delta)}/t_{1}t_{3};q\right)_{\infty}}\frac{\left(-q^{2+n(\beta+\gamma)+k}/t_{1}t_{2};q\right)_{\infty}}{\left(-q^{2+n(\beta+\gamma)}/t_{1}t_{2};q\right)_{\infty}}\\
	& =1+\mathcal{O}\left(q^{n-\sqrt{n}}\right),\quad n\to\infty,
	\end{align*}
	then,
	\begin{align}
	& S_{1}(n)=\sum_{k=0}^{\left\lfloor \sqrt{n}\right\rfloor }\frac{q^{k^{2}}}{\left(q;q\right)_{k}}\left(\frac{-q^{2}z}{t_{1}t_{2}t_{3}}\right)^{k}+\mathcal{O}\left(q^{\eta n}\right)\label{eq:qairy-5}\\
	& =A_{q}\left(\frac{q^{2}z}{t_{1}t_{2}t_{3}}\right)-\sum_{k=\left\lfloor \sqrt{n}\right\rfloor +1}^{n}\frac{q^{k^{2}}}{\left(q;q\right)_{k}}\left(\frac{-q^{2}z}{t_{1}t_{2}t_{3}}\right)^{k}+\mathcal{O}\left(q^{\eta n}\right)\nonumber \\
	& =A_{q}\left(\frac{q^{2}z}{t_{1}t_{2}t_{3}}\right)+\mathcal{O}\left(q^{\eta n}\right),\quad n\to\infty,\nonumber 
	\end{align}
	where
	\begin{equation}
	0<\eta<\min\left\{ 1,\,(\alpha+\beta),\,(\beta+\delta),\,(\beta+\gamma)\right\} ..\label{eq:qairy-6}
	\end{equation}
	It is clear that 
	\begin{equation}
	S_{1}(n)=\mathcal{O}\left(q^{\eta n}\right),\ \frac{\left(-q^{2+n(\beta+\delta)}/t_{1}t_{3};q\right)_{n}}{\left(-q^{2+n}/t_{2}t_{3};q\right)_{n}}=\mathcal{O}\left(q^{\eta n}\right),\quad n\to\infty,\label{eq:qairy-7}
	\end{equation}
	and Theorem \ref{thm:qairy} is obtained by combining \eqref{eq:qairy-5},\eqref{eq:qairy-6}
	and \eqref{eq:qairy-7}. 
\end{proof}

\begin{thm}
	\label{thm:theta}For any 
	\begin{equation}
	0<\alpha<1,\quad w\notin\left\{ q^{-(2n-1)/2}\vert n\in\mathbb{Z}\right\} \cup\left\{ 0\right\} \label{eq:theta-1}
	\end{equation}
	we have
	\begin{equation}
	\left(\frac{\sqrt{q}}{i}\right)^{n}V_{n}\left(\frac{w+w^{-1}}{2};q^{\frac{1}{2}},\,q^{-n\alpha},\,q^{-n\alpha}\bigg|q\right)=\theta_{4}\left(w;q^{\frac{1}{2}}\right)\left\{ 1+\mathcal{O}\left(q^{n\alpha}\right)\right\} ,\label{eq:theta-2}
	\end{equation}
	as $n\to\infty$, where 
	\begin{equation}
	\theta_{4}(w;q)=\left(q^{2},\,q/w,\,qw;q^{2}\right)_{\infty}=\sum_{n=-\infty}^{\infty}(-1)^{n}q^{n^{2}}w^{n}.\label{eq:theta-3}
	\end{equation}
\end{thm}

\begin{proof}
	First we observe that
	\begin{align*}
	& V_{n}(x;\,\mathbf{t}|q)=\frac{t_{1}^{n}\left(-q^{2}/t_{1}t_{3};q\right)_{n}}{q^{n}\left(-q^{2}/t_{2}t_{3};q\right)_{n}}{}_{3}\phi_{2}\begin{pmatrix}\begin{array}{c}
	q^{-n},-q/t_{1}z,\,qz/t_{1}\\
	-q^{2}/t_{1}t_{3},-q^{2}/t_{1}t_{2}
	\end{array} & \bigg|q,\,-\frac{q^{n+2}}{t_{2}t_{3}}\end{pmatrix}\\
	& =\frac{t_{1}^{n}\left(-q^{2}/t_{1}t_{3};q\right)_{n}}{q^{n}\left(-q^{2}/t_{2}t_{3};q\right)_{n}}\sum_{k=0}^{n}\frac{\left(q^{-n},-q/t_{1}z,\,qz/t_{1};q\right)_{k}}{\left(q,-q^{2}/t_{1}t_{3},-q^{2}/t_{1}t_{2};q\right)_{k}}\left(-\frac{q^{n+2}}{t_{2}t_{3}}\right)^{k}\\
	& =\frac{t_{1}^{n}\left(q,-q^{2}/t_{1}t_{3};q\right)_{n}}{q^{n}\left(-q^{2}/t_{2}t_{3};q\right)_{n}}\sum_{k=0}^{n}\frac{q^{k^{2}/2+k/2}}{\left(q;q\right)_{k}\left(q;q\right)_{n-k}}\left(\frac{q}{t_{2}t_{3}}\right)^{k}\frac{\left(-q/t_{1}z,\,qz/t_{1};q\right)_{k}}{\left(-q^{2}/t_{1}t_{3},-q^{2}/t_{1}t_{2};q\right)_{k}}\\
	& =\frac{t_{1}^{n}\left(q,-q^{2}/t_{1}t_{3};q\right)_{n}}{q^{n}\left(-q^{2}/t_{2}t_{3};q\right)_{n}}\sum_{k=0}^{n}\frac{q^{k^{2}/2+k/2}}{\left(q;q\right)_{k}\left(q;q\right)_{n-k}}\left(\frac{q}{t_{2}t_{3}}\right)^{k}\\
	& \times\frac{\left(-q/t_{1}z,\,qz/t_{1};q\right)_{\infty}}{\left(-q^{k+1}/t_{1}z,\,q^{k+1}z/t_{1};q\right)_{\infty}}\frac{\left(-q^{2+k}/t_{1}t_{3},-q^{2+k}/t_{1}t_{2};q\right)_{\infty}}{\left(-q^{2}/t_{1}t_{3},-q^{2}/t_{1}t_{2};q\right)_{\infty}}\\
	& =\frac{t_{1}^{n}\left(q,-q^{2}/t_{1}t_{3};q\right)_{n}}{q^{n}\left(-q^{2}/t_{2}t_{3};q\right)_{n}}\frac{\left(-q/t_{1}z,\,qz/t_{1};q\right)_{\infty}}{\left(-q^{2}/t_{1}t_{3},-q^{2}/t_{1}t_{2};q\right)_{\infty}}\\
	& \times\sum_{k=0}^{n}\frac{q^{k^{2}/2+k/2}}{\left(q;q\right)_{k}\left(q;q\right)_{n-k}}\left(\frac{q}{t_{2}t_{3}}\right)^{k}\frac{\left(-q^{2+k}/t_{1}t_{3},-q^{2+k}/t_{1}t_{2};q\right)_{\infty}}{\left(-q^{k+1}/t_{1}z,\,q^{k+1}z/t_{1};q\right)_{\infty}}.
	\end{align*}
	Let 
	\begin{equation}
	t_{2}=t_{3}=q^{-n\alpha},\quad0<\alpha<1\label{eq:theta-4}
	\end{equation}
	and
	\begin{equation}
	z\notin\left\{ -q^{n}/t_{1}\vert n\in\mathbb{N}\right\} \cup\left\{ t_{1}q^{-n}\vert n\in\mathbb{N}\right\} .\label{eq:theta-5}
	\end{equation}
	As $n\to\infty$, since
	\begin{align*}
	& \sum_{0\le k\le\left\lfloor \sqrt{n}\right\rfloor }\frac{q^{k^{2}/2+k/2}}{\left(q;q\right)_{k}\left(q;q\right)_{n-k}}\left(\frac{q}{t_{2}t_{3}}\right)^{k}\frac{\left(-q^{2+k}/t_{1}t_{3},-q^{2+k}/t_{1}t_{2};q\right)_{\infty}}{\left(-q^{k+1}/t_{1}z,\,q^{k+1}z/t_{1};q\right)_{\infty}}\\
	& =\sum_{0\le k\le\left\lfloor \sqrt{n}\right\rfloor }\frac{q^{k^{2}/2+3k/2+2n\alpha}}{\left(q;q\right)_{k}\left(q;q\right)_{n-k}}\frac{\left(-q^{2+k+n\alpha}/t_{1},-q^{2+k+n\alpha}/t_{1};q\right)_{\infty}}{\left(-q^{k+1}/t_{1}z,\,q^{k+1}z/t_{1};q\right)_{\infty}}\\
	& =\left(1+\mathcal{O}\left(q^{n-\sqrt{n}}\right)\right)\sum_{0\le k\le\left\lfloor \sqrt{n}\right\rfloor }\frac{q^{k^{2}/2+3k/2+2n\alpha}}{\left(q;q\right)_{k}\left(-q^{k+1}/t_{1}z,\,q^{k+1}z/t_{1};q\right)_{\infty}}\\
	& =\left(1+\mathcal{O}\left(q^{n-\sqrt{n}}\right)\right)\left(1+\mathcal{O}\left(q^{n\alpha}\right)\right)=1+\mathcal{O}\left(q^{n\alpha}\right),
	\end{align*}
	and
	\[
	\sum_{\sqrt{n}<k\le n}\frac{q^{k^{2}/2+k/2}}{\left(q;q\right)_{k}\left(q;q\right)_{n-k}}\left(\frac{q}{t_{2}t_{3}}\right)^{k}\frac{\left(-q^{2+k}/t_{1}t_{3},-q^{2+k}/t_{1}t_{2};q\right)_{\infty}}{\left(-q^{k+1}/t_{1}z,\,q^{k+1}z/t_{1};q\right)_{\infty}}=\mathcal{O}\left(q^{n\alpha}\right),
	\]
	then,
	\begin{equation}
	\sum_{k=0}^{n}\frac{q^{k^{2}/2+k/2}}{\left(q;q\right)_{k}\left(q;q\right)_{n-k}}\left(\frac{q}{t_{2}t_{3}}\right)^{k}\frac{\left(-q^{2+k}/t_{1}t_{3},-q^{2+k}/t_{1}t_{2};q\right)_{\infty}}{\left(-q^{k+1}/t_{1}z,\,q^{k+1}z/t_{1};q\right)_{\infty}}=1+\mathcal{O}\left(q^{n\alpha}\right).\label{eq:theta-6}
	\end{equation}
	This together with
	\begin{equation}
	\frac{\left(q,-q^{2}/t_{1}t_{3};q\right)_{n}}{\left(-q^{2}/t_{2}t_{3};q\right)_{n}}=\frac{\left(q,-q^{2+n\alpha}/t_{1};q\right)_{n}}{\left(-q^{2+n\alpha}/t_{2};q\right)_{n}}=\left(q;q\right)_{\infty}\left(1+\mathcal{O}\left(q^{n\alpha}\right)\right)\label{eq:theta-7}
	\end{equation}
	and 
	\begin{equation}
	\frac{1}{\left(-q^{2}/t_{1}t_{3},-q^{2}/t_{1}t_{2};q\right)_{\infty}}=\frac{1}{\left(-q^{2+n\alpha}/t_{1},-q^{2+n\alpha}/t_{1};q\right)_{\infty}}=1+\mathcal{O}\left(q^{n\alpha}\right)\label{eq:theta-8}
	\end{equation}
	gives
	\begin{equation}
	\left(\frac{q}{t_{1}}\right)^{n}V_{n}(x;\,\mathbf{t}|q)=\left(q,-q/t_{1}z,\,qz/t_{1};q\right)_{\infty}\left\{ 1+\mathcal{O}\left(q^{n\alpha}\right)\right\} \label{eq:theta-9}
	\end{equation}
	as $n\to\infty$. Theorem \ref{thm:theta} is proved by taking $t_{1}=i\sqrt{q}$
	and $z=iw$.
\end{proof}

\noindent{\bf Acknowledgements} This paper started when Ismail and Zhang  
were visiting East China Normal University in Shanghai hosted by 
Zhi-Guo Liu. We appreciate the hospitality and the great research environment provided by Liu and the first named author also acknowledge the financial support for his trip.

\noindent M. E. H. I, 
University of Central Florida, Orlando, Florida 32816\\
  email: ismail@math.ucf.edu
  \bigskip
   
\noindent R. Z.,   College of Science, 
Northwest A\&F University, 
Yangling, Shaanxi 712100, 
P. R. China PRC\\
   email:  ruimingzhang@yahoo.com

  \bigskip 
 \noindent K. Z.,
 School of Mathematics and Statistics, Central South University, Changsha, Hunan 410083, P.R. China.\\
 email: krzhou1999@knights.ucf.edu

\end{document}